\documentclass[10pt,reqno]{amsart}
 \usepackage{hyperref}\hypersetup{colorlinks=true, citecolor=blue}
\usepackage{xfrac}
\usepackage{amsfonts,amsmath,amssymb,epsfig,mathtools}
\usepackage{esint,ifthen,color,mathrsfs,bm}
\usepackage{a4wide}

\newtheorem{theorem}{Theorem}[section]
\newtheorem{lemma}[theorem]{Lemma}
\newtheorem{proposition}[theorem]{Proposition}
\newtheorem{corollary}[theorem]{Corollary}

\theoremstyle{definition}

\theoremstyle{remark}

\theoremstyle{remark}

\numberwithin{equation}{section}

\newcommand{\R}{\mathbb{R}}
\newcommand{\eps}{{\varepsilon}}
\newcommand{\Lip}{{\mathrm{Lip}}}

\newcommand{\dist}{{\textup {dist}}}

\newcommand{\de}{\partial}

\newcommand{\weak}{{\rightharpoonup}\,}

\renewcommand{\d}{{\rm d}}
\newcommand{\Id}{{\rm Id}}
\newcommand{\spt}{{\rm spt}\,}

\newcommand{\loc}{\textup{loc}}

\newcommand{\gr}{\textup{G}}

\renewcommand{\and}{\quad \text{and} \quad}

\renewcommand{\div}{\textup{div}}

\newcommand{\cL}{{\mathcal{L}}}
\newcommand{\cH}{{\mathcal{H}}}

\newcommand{\cT}{{\mathcal{T}}}

\newcommand\N{{\mathbb N}}

\newcommand{\ie}{{\textit{i.e.}}}
\newcommand{\eg}{{\textit{e.g.}}}

\def\mc#1{\div\Big(\frac{\nabla {#1}}{\sqrt{1+|\nabla {#1}|^2}}\Big)}
\def\jac#1{\sqrt{1+|\nabla {#1}|^2}}

\title{How a minimal surface leaves a thin obstacle}

\author[M.~Focardi]{Matteo Focardi}
\address{DiMaI, Universit\`a degli Studi di Firenze}
\curraddr{Viale Morgagni 67/A, 50134 Firenze (Italy)}
\email{matteo.focardi@unifi.it}
\author[E.~Spadaro]{Emanuele Spadaro}
\address{La Sapienza Universit\`a di Roma}
\curraddr{Piazzale Aldo Moro 5, 00185 Roma (Italy)}
\email{emanuele.spadaro@mat.uniroma1.it}
\thanks{
E.~S.~has been supported by the ERC-STG Grant n. 759229
HiCoS
``Higher Co-dimension Singularities: Minimal Surfaces and 
the Thin Obstacle Problem''.
E.\ S. has been also partially supported
by the Gruppo Nazionale per l'Analisi Matematica, 
la Probabilit\`a e le loro Applicazioni (GNAMPA) of the 
Istituto Nazionale di 
Alta Matematica (INdAM) through a Visiting Professor Fellowship. 
E.~S.~is very grateful to the DiMaI ``U. Dini'' of the University
of Firenze for the support during the visiting period. M.~F.
is a member of the GNAMPA of INdAM.}

\subjclass[2010]{35R35, 49Q05}

\keywords{Minimal immersion, thin obstacle problem, free boundary, 
2-valued functions}

\date{}
\date{}
\begin{document}
\begin{abstract}
We prove the optimal regularity and a detailed analysis
of the free boundary of the solutions to the thin
obstacle problem for nonparametric minimal surfaces with
flat obstacles. 
\end{abstract}

\maketitle

%
%
\section{Introduction}\label{s:Intro}
The present note focuses on the analysis of
the thin obstacle problem for nonparametric minimal surfaces.
This is a well-known variational problem which has been
extensively considered in the literature,
cf.~the classical works by Nitsche \cite{Nitsche69}, 
Giusti \cite{Giu71, Giu72, GiuCIME, Giu74}, Kinderlehrer \cite{Kind71} and Frehse \cite{Freh75, Freh77}.
In this respect, the vast literature on thin obstacle problems
with quadratic energies, which correspond to the linearization
of the area functional, has to be taken into account. 
Starting with the pioneering contributions by Lewy \cite{Lewy70, Lewy72},
Richardson \cite{Rich78}, Caffarelli \cite{Caf79},
Kinderlehrer \cite{Kind81}, and Ural'tseva \cite{Ural85,Ural86,Ural87},
a renewed impulse towards a deeper understanding 
of the problem has started more recently with the works of Athanasopoulos and 
Caffarelli \cite{AtCa04}, Athanasopoulos, Caffarelli and Salsa 
\cite{AtCaSa08}, Caffarelli, Salsa and Silvestre \cite{CaSaSi08}
and has been then developed by many others
\cite{FoSp16, FoSp17, Guillen14, GaSm14, KRS17} 
etc\ldots\,
we warn the readers that this is only a small excerpt
from the literature on the topic.
To complete the overview on the topic we also mention the parametric approach to minimal surfaces with thin obstacles, which has been started by De~Giorgi (identifying the relaxation of the problem via the introduction of the nowadays called De~Giorgi's measure) and developed in the monograph by De~Giorgi, Colombini and Piccinini \cite{DeGColPic72}, and then in the papers by De~Giorgi \cite{DeG73} and by De~Acutis \cite{Mimma79}. Very recently it has been further extended by 
Fern\'andez-Real and Serra \cite{FeSe17}.

Despite the nonlinear thin obstacle problem naturally arises
in several applications and has attracted the attention
of distinguished mathematicians,
some of the most important questions concerning the
regularity of the solutions remained unsolved for many years.
For partial results in this regards, 
aside from the quoted papers by Nitsche, Giusti, Frehse and Ural'tseva
on nonlinear variational operators,
we mention the more recent contributions by Milakis and Silvestre
\cite{MiSi}, Fern\'andez-Real \cite{Fe16}, Ros-Oton and Serra 
\cite{ROSe} in the fully nonlinear case.

Building upon the works by Frehse \cite{Freh77} and
Ural'tseva \cite{Ural86} together with our previous work \cite{FoSp18}, 
in the present paper we give the first comprehensive analysis
in the relevant geometric setting of nonparametric minimal surfaces with 
thin obstacles, developing
an approach which can be further extended
to more general nonlinear operators.
For the sake of simplicity,
we confine ourselves to the following elementary
formulation of the thin obstacle problem 
for the nonparametric area functional:
given $g\in C^2(\R^{n+1})$
satisfying  $g\vert_{\R^n\times \{0\}} \geq 0$
and $g(x',x_{n+1}) = g(x', -x_{n+1})$, we consider the 
variational problem
\begin{equation}\label{e:ob prob}
\min_{v \in \mathscr{A}_{g}} \quad
\int_{B_1} \sqrt{1+|\nabla v|^2} \,\d x
\end{equation}
where $\mathscr{A}_{g}:=\big\{v\in g\vert_{B_1} + 
W^{1,\infty}_0(B_1):\; v\vert_{B_1'}\geq 0, \, v(x',x_{n+1}) = 
v(x',-x_{n+1})\big\}$.  Here $B_1'=B_1\cap\{x_{n+1}=0\}$,
in addition we set $B_1^+:=B_1\cap\{x_{n+1}>0\}$.
As reported right below, the assumption of flat obstacles allows to solve the problem in the space of Lipschitz functions, 
while for non-flat obstacle the right space to work with is that of functions of bounded variation. Part of the results of this paper 
can be generalized to non-flat and non-zero obstacles (see, e.g., the
techniques in our paper \cite{FoSp18} on the fractional obstacle problem), but at the best of our knowledge a complete analysis in the 
general case is still missing.

Existence and uniqueness of a solution $u$ in the class
$g\vert_{B_1} + W^{1,\infty}_0(B_1)$ has been established
by Giusti \cite{Giu71, Giu72, GiuCIME, Giu74}
(following the analysis of minimal surfaces with classical obstacles
by Giaquinta and Pepe \cite{GiaPe71} -- see also Giaquinta and Modica
\cite{GiaMo75}), showing
that $u$ can be characterized as the weak solution to the system:
\begin{gather}\label{e.signorini bc}
\begin{cases}
\mc{u} = 0 & \text{in }\; B_1^+,\\
\de_{n+1} u \leq 0 \quad\text{and}\quad
u\,\de_{n+1} u = 0 & \text{on }\; B_1'.
\end{cases}
\end{gather}
Lipschitz continuity for $u$ is the best possible global regularity
in $B_1$, as simple examples show. Nevertheless,
the solution is expected to be more regular on both sides of
the obstacle, thus leading to the investigation of
the one-sided regularity on $B_1^+\cup B_1'$.
This is a central question in understanding the qualitative
properties of the solutions to variational inequalities with
thin obstacles and several important results have been achieved
in the last decades.
The first contributions to this issue were given 
by H. Lewy in the two dimensional setting \cite{Lewy70,Lewy72}. 
Lately, continuity of the first derivatives of $u$ taken along 
tangential directions to $B_1'$ in any dimension and one-sided
continuity (up to $B_1'$) for the normal derivative in two dimensions
({\ie} $n=1$) were obtained by Frehse \cite{Freh75,Freh77}
for solutions to very general variational inequalities.
On the other hand, 
for the corresponding problem in the
uniformly elliptic setting, more refined results
on the one-sided regularity are available: in particular, 
the H\"older continuity of the derivatives, firstly established by Richardson \cite{Rich78} in dimension two and by 
Caffarelli \cite{Caf79} in any dimension, is shown by different proofs and in different degrees of generality, see
\cite{Kind81,Ural85,Ural86,Ural87,AtCa04,GaSm14,Guillen14,KRS17}
only to mention few references.

Despite all the mentioned recent achievements, 
for the geometric nonlinear
case of nonparametric minimal surfaces the $C^{1,\alpha}$
one-sided regularity of solutions was not known in general
(except for the two dimensional case
considered by Frehse \cite{Freh75} and more recently by 
Fern\'andez-Real and Serra \cite{FeSe17}\footnote{
After the appeareance of this manuscript, in the second version of the preprint \cite{FeSe17} the authors establish the almost optimal regularity in any dimension, proving that the solutions to the parametric thin obstacle problem for Caccioppoli sets are $C^{1,\sfrac12-\eps}$ regular for every $\eps>0$. This improvement gives a different proof of the non-optimal $C^{1,\sfrac12-\eps}$ regularity provided in this note.
}).
In this paper we establish the first result on the optimal $C^{1,\sfrac12}$ 
regularity (to the best of our knowledge there are no other examples
of optimal regularity for nonlinear variational inequalities with
thin obstacles) and we provide a detailed analysis of
the free boundary of the coincidence set.
Our approach is based on the pioneering analysis by Frehse
\cite{Freh75,Freh77}, by Uralt'seva \cite{Ural85, Ural86, Ural87}
and on our previous analysis of the Signorini problem
\cite{FoSp16, FoSp18}.
Starting from these results, we develop here a blowup analysis for
the study of nonparametric minimal surfaces with
thin obstacle, which can be further extended to other 
nonlinearities. In particular, we do not use the optimal regularity
for the scalar Signorini problem established in \cite{AtCa04}, 
but we can actually reprove it easily adapting the arguments of the present note.

The following is the main result of the paper 
(actually, more refined conclusions
will be shown, cf. the statement of Theorem~\ref{t:main precise}).

\begin{theorem}\label{t.main}
Let $u$ be a solution to the thin obstacle problem
\eqref{e:ob prob} and let $\Gamma(u)$ be its free boundary,
namely the boundary of $\{(x',0)\in B_1':\,u(x',0)=0\}$ in 
the relative topology of $B_1'$. Then, 
\begin{itemize}
\item[(i)] $u\in C^{1,\sfrac12}_{\loc}(B_1^+\cup B_1')$;
\item[(ii)] $\Gamma(u)$ has locally finite $(n-1)$-dimensional Hausdorff measure
and is $\cH^{n-1}$-rectifiable.
\end{itemize}
\end{theorem}

More in details, concerning the proof of the results we
proceed in several steps.
Complementing Frehse's result \cite{Freh77}, 
we establish first in Section~\ref{s:regolarita'} the one-sided 
$C^1$-smoothness of the normal derivative of the solution $u$.
Then, we show the  H\"older continuity of the first derivatives (one-sided for the normal one) in Section~\ref{s:C1alpha}. 
In doing this we use a penalization argument together with the celebrated De Giorgi's method to prove H\"older regularity,
following the approach outlined by Ural'tseva \cite{Ural86} 
in the strongly elliptic case.  
Optimal regularity then follows by an interesting connection 
with the theory of minimal surfaces highlighted in 
Section~\ref{s:optimal}. More precisely, we show that 
solutions to the thin obstacle problem for the area functional
correspond to two-valued minimal graphs. Given this, we can exploit
the recent results by Simon and Wickramasekera \cite{SiWi16} 
to infer the optimal one-sided $C^{1,\sfrac12}$ regularity.
This association links thin obstacle problems with the program started by Krummel and Wickramasekera \cite{KW13} about the regularity of multiple valued solutions to the minimal surface system. In this regards, the results in \cite{KW13} are mostly concerned with the regularity of harmonic multiple valued functions (see also \cite{DS-memo, FMS-15, DMSV17} for more other results), while their extension to the minimal surface system are not yet known: further investigations in this direction are needed to extend the approach developed here and in \cite{DMSV17,FoGeSp15,FoSp17} to prove the regularity of multiple valued minimal graphs.

In the last section of the paper, we consider
the free boundary analysis, i.e.
the study of the measure theoretic and geometric properties 
of the \textit{free boundary} set $\Gamma(u)$, defined as
the topological boundary in the relative topology of 
$B_1'$ of the \textit{coincidence set} $\Lambda(u) = 
\big\{(x', 0) \in B_1'\,: u(x', 0) =0 \big\}$.
In this respect we follow our recent paper on the Signorini
problem for the fractional Laplacian \cite{FoSp17,FoSp17bis,FoSp18}
and show the $\mathcal{H}^{n-1}$-rectifiability of the free boundary
and the local finiteness of its Hausdorff measure (actually of its
Minkowski content).
In Section~\ref{s:freebdry} we provide the essential key tools
to follow the strategy developed in \cite{FoSp17,FoSp18}. In particular, 
we prove a quasi-monotonicity formula for the Almgren's type
frequency function 
\begin{equation}\label{e:freq intro}
I_u(x_0,r) := 
\frac{r \int \phi\big(\textstyle{\frac{|x-x_0|}{r}}\big)\frac{|\nabla u|^2}{\sqrt{1+|\nabla u|^2}}\d x}
{- \int \phi'\Big(\textstyle{\frac{|x-x_0|}{r}}\Big)\,\frac{1}{|x-x_0|}\frac{u^2}{\sqrt{1+|\nabla u|^2}}\,\d x}
\end{equation}
for $r<1-|x_0|$ and $x_0\in B_1'$ (see Section~\ref{s:frequency}
for the definition of the auxiliary function $\phi$ and the details).

%
%
\section{Preliminaries}\label{s.preliminaries}
Throughout the paper we use the following notation:
for any subset $E \subset \R^{n+1}$ we set
\[
E^\pm := E\cap\big\{x\in\R^{n+1}:\,\pm x_{n+1}>
0\big\}
\and
E' := E \cap \big\{x_{n+1}=0\big\}.
\]
For $x \in \R^{n+1}$ we write $x = (x', x_{n+1})\in \R^{n} 
\times \R$ and
$B_r(x) \subset \R^{n+1}$ denotes the open ball centered
at $x\in \R^{n+1}$ with radius $r>0$ 
(we omit to write the point $x$ if the origin and, 
when there is no source of ambiguity, 
we write $x'$ for the point $(x',0)$).

In what follows we shall use the terminology \emph{solution of the thin obstacle problem}
for a minimizer $u$ of the area funtional on $B_1^+$ 
with respect to its own boundary conditions and additionally satisfying the unilateral obstacle constraint $u|_{B_1'}\geq 0$.

We recall the following two results which will be used in the sequel.
\begin{proposition}\label{p.comparison}
Let $u$ and $v\in W^{1, \infty}(B_1)$ be two solutions to the thin obstacle
problem. 
If $u\vert_{\de B_1} \leq v\vert_{\de B_1}$, then $u\leq v$ on $\bar{B_1}$.
\end{proposition}

The proof is a direct consequence of the comparison principle for
minimal surfaces (cf. \cite[Chapter 1, Lemma 1.1]{Giusti03}).

The second result we need is due to Frehse \cite{Freh77}.
In order to state it, we introduce the following
general formulation: let
$F:\R^{n+1}\times \R\times \R^{n+1} \to \R$ be a smooth function
(we denote its variables by $(x,z, p)$) and consider the corresponding functional
\[
\mathcal{F}(u):=\int_{B_1} F(x,u(x), \nabla u(x))\, \d x.
\]
We assume that the Hessian matrix $\big(\frac{\de^2 F}{\de p_i\de p_j}\big)_{i,j=1, \ldots, n+1}$ of $F$ 
is uniformly elliptic ({\ie} uniformly positive definite) and bounded.
The thin obstacle problem related to $F$ is then obtained by minimizing $\mathcal{F}$ among all functions in $\mathcal{A}_g$.

\begin{theorem}[\cite{Freh77}]\label{t.frehse}
Under the assumptions above on $F$, the Lipschitz solutions $u$ to the 
corresponding thin obstacle problems satisfy: 
\begin{itemize}
\item[(i)] if $n=1$, then $u \in C^1(B_1^+\cup B_1')$ with
\[
|\nabla u (x) - \nabla u(y)| \leq \omega_0(|x-y|) \quad \forall\;
x,y\in B_1^+\cup B_1'
\]
where $\omega_0(t)= C|\log t|^{-q}$ with $q\geq 0$ is any constant and $C>0$;

\item[(ii)] if $n\geq 2$, then the tangential derivatives 
$\de_i u \in C^0(B_1^+\cup B_1')$ for $i \in\{1, \ldots, n\}$ with
\[
|\de_i u (x) - \de_i u(y)| \leq \omega_1(|x-y|) \quad \forall\;
x,y\in B_1^+\cup B_1',
\]
where $\omega_1(t) = C|\log t|^{-q(n)}$ with $q(n) \in (0, \frac{2}{(n+1)^2-2n-2})$
and $C>0$.
\end{itemize}
\end{theorem}

%
%
\section{$C^{1}$ regularity}\label{s:regolarita'}
The existence, uniqueness and the Lipschitz 
regularity of the solutions
to the variational problem \eqref{e:ob prob} have been studied in 
\cite{Giu71,Giu72,GiuCIME}.
In this section we show that the solutions to the thin 
obstacle problem have one-sided continuous derivative.
In two dimension, this result is due to Frehse \cite{Freh77}
for general nonlinear variational inequalities.
In higher dimensions, this is not known in this generality
and here we provide a proof for the specific case
of the area functional.

\begin{proposition}\label{p.C1}
Let $u\in W^{1,\infty}(B_1)$ be a solution to the thin
obstacle problem. Then, $u
\in C^1(B_1^+\cup B_1')$.
\end{proposition}

For the proof of the proposition we start with the
following two lemmas.

\begin{lemma}\label{l.barrier}
For every $a>0$ there exists $\eps>0$
such that the solution $w_\eps:B_1\to \R$ to the thin obstacle problem 
with boundary value $g_\eps(x) = -a |x_{n+1}| + \eps$
satisfies
\begin{equation}\label{e.barriera}
w_\eps\vert_{B_{\sfrac34}'} \equiv 0.
\end{equation}
\end{lemma}

\begin{proof}
From the uniqueness of the solutions to the 
obstacle problems \eqref{e:ob prob}
and the radial symmetry
of the boundary value $g_\eps(x', x_{n+1}) = g_\eps(y', x_{n+1})$
if $|x'| = |y'|$, we deduce that
$w_\eps(x) = \phi_\eps(|x'|, x_{n+1})$
for some function $\phi_\eps: B_1\subset \R^2 \to \R$.
Moreover, from the regularity of $w_\eps$ (see, e.g., \cite[Theorem 4]{Giu74}) and from its variational
characterization, it follows that
$\phi_\eps$ is locally Lipschitz and solves
the variational problem
\begin{equation}\label{e.problema bidim}
\phi_\eps \in \text{argmin}_{\phi\in \mathcal{C}} \int_{B_1}
\sqrt{1 + |\nabla \phi(\rho,t)|^2}
\,\rho^{n-1}\, \d \rho \d t
\end{equation}
with
\[
\mathcal{C}:=\big\{
\phi\vert_{\de B'_1}\geq 0 \quad \text{and}\quad
\phi(\rho,t) = -a|t| + \eps \quad \forall\; (\rho, t)\in \de B_1
\big\}.
\]
In particular, from Theorem \ref{t.frehse} (i)
it follows that where the integrand is uniformly elliptic,
the solutions $\phi_\eps$ have uniform continuity bounds 
on their derivatives. Thus, in particular, 
\[
\vert
\nabla \phi_\eps(x) - \nabla \phi_\eps (y) \vert
\leq \omega_0(|x-y|)
\quad \forall\; x,y \in {B_{\sfrac34}^+\setminus B_{\sfrac14}^+},
\]
where $\omega_0$ is the modulus of continuity
in Theorem \ref{t.frehse} (i).
In particular, from Proposition \ref{p.comparison} it follows that 
$w_\eps$ converge in $C^1(B_{\sfrac34}^+\setminus B_{\sfrac14}^+)$
to $w_\infty(x) := - a x_{n+1}$ and
\begin{equation}\label{e:derivata negativa0}
\lim_{\eps \to 0}\|\de_{n+1}w_\eps 
+a\|_{L^{\infty}(B_{\sfrac34}^+\setminus B_{\sfrac14}^+)}
= 0.
\end{equation}
We then infer that there exists $\eps_0>0$ such that
\[
\de_{n+1} w_\eps (x)
\leq - \sfrac{a}2
\quad\forall\; \eps \in (0, \eps_0), \;\forall\;
x \in B_{\sfrac34}^+\setminus B_{\sfrac14}^+, 
\]
and in view of Theorem~\ref{t.frehse} (i)
\begin{equation}\label{e:derivata negativa}
\de_{n+1} w_\eps (x',0^+)
:= \lim_{t\to 0^+}\frac{w_\eps(x',t) - w_\eps(x',0)}{t}
\leq - \sfrac{a}2
\end{equation}
for $\eps \in (0, \eps_0)$
and $x' \in B_{\sfrac34}'\setminus B_{\sfrac14}'$.
Recalling the Euler--Lagrange equations associated
to the thin obstacle problem \eqref{e.signorini bc},
this implies that $B_{\sfrac34}'\setminus B_{\sfrac14}' 
\subset \Lambda(w_\eps)$
for all $\eps<\eps_0$.

We need only to show that $B_{\sfrac14}'\subset
\Lambda(w_\eps)$ if $\eps$
is suitably chosen. To this aim we show that, 
for $\eps$ sufficiently small, we have that
\begin{equation}\label{e.maggiorazione al bordo}
\phi_\eps (\rho, t) \leq - \frac{a}{2}\, t \quad 
\forall\; (\rho,t) \in
\de B_{\sfrac12}.
\end{equation}
Indeed, given for granted the last inequality, 
the comparison principle for the solutions to the thin
obstacle problem in Proposition \ref{p.comparison},
yields 
that $w_\eps (x) \leq - \frac{a}{2}\,|x_{n+1}|$
for every $x\in \bar{B}_{\sfrac12}$, from which 
$B_{\sfrac14}'\subset \Lambda(w_\eps)$ readily follows.
In order to show \eqref{e.maggiorazione al bordo}, 
we notice that by \eqref{e:derivata negativa0}
\[
\phi_\eps(\rho,t) \leq - \frac{a}{2}\,t \quad 
\quad\forall\;\eps \in (0, \eps_0),\;
\forall\;\rho \in (\sfrac14,\sfrac12)
\;\text{and}\;\forall\;
t\in (0, \sfrac{\sqrt{5}}{4}), 
\]
where we used that $(x',t)\in B_{\sfrac34}^+\setminus B_{\sfrac14}^+$
if $|x'| \in (\sfrac14,\sfrac12)$ and 
$t\in (0, \sfrac{\sqrt{5}}{4})$.
Moreover, since $\phi_\eps$ converges to $-a t$ in $B_{\sfrac34}^+
\setminus B_{\sfrac14}^+$, we also infer that there exists
$\eps_1>0$ such that
\[
\phi_\eps (\rho, t) \leq -a t + \frac{a}{8}\leq
- \frac{a}{2}\, t\;
\quad\forall\;\eps \in (0, \eps_1),\;
\; \forall\; (\rho,t) \in
\de B_{\sfrac12}^+, \; t\geq \sfrac14. 
\]
Putting the two estimates together, we deduce that
\eqref{e.maggiorazione al bordo} holds
for every $\eps < \min\{\eps_0, \eps_1\}$,
thus concluding the lemma.
\end{proof}

We prove next an auxiliary result.

\begin{lemma}\label{l.blowup}
Let $u\in W^{1,\infty}(B_1)$ be a solution to the thin
obstacle problem \eqref{e:ob prob}.
Then, for any sequence of points $z_k\in \Gamma(u)$
and of radii $t_k\downarrow 0$ (with $t_k\leq 1-|z_k|$),
the functions
\[
u_k (x) := \frac{u(z_k+t_k x)}{t_k}
\]
converge to $0$ uniformly on $\bar B_1$.
\end{lemma}

\begin{proof}
The functions $u_k$ are equi-Lipschitz continuous
(with $\Lip(u_k)\leq \Lip(u)$) and are 
solutions to the thin obstacle problem with $\underline{0}\in \Gamma(u_k)$.
Therefore, up to passing to a subsequence
(not relabeled for convenience),
$u_k$ converges uniformly on $\bar B_1$ to a function $u_\infty$
which is itself a solution to the thin obstacle problem.
We need now to prove that $u_\infty\equiv 0$.

We start noticing that, in view of Theorem~\ref{t.frehse} (ii),
we have
\begin{equation}\label{e:ukequi}
|\nabla' u_k(x) - \nabla' u_k(y)| =
|\nabla' u(t_k x+z_k) - \nabla' u(t_k y +z_k)| 
\leq \omega_1(t_k |x-y|),
\end{equation}
where $\nabla' = (\de_1, \ldots, \de_n)$ denotes the horizontal
gradient.
Thus, by \eqref{e:ukequi} and since $\nabla' u_k(\underline{0}) = \underline{0}$, 
$\|\nabla' u_k\|_{\infty}$ converge to $0$.
Being $\nabla' u_k$ equi-continuous (with modulus of continuity 
$\omega_1$), we then infer that $\nabla' u_k$
converges to $\nabla' u_\infty$ uniformly and
$\nabla' u_\infty \equiv \underline{0}$, {\ie}
$u_\infty$ is a function depending exclusively on the 
variable $x_{n+1}$. By direct computation one 
can show that the only solutions depending on one 
variable are the linear functions of the form
\[
u_\infty (x) = -a x_{n+1}\quad\text{on $\bar B_1^+$, for some $a\geq 0$}.
\]
The thesis is then reduced to proving that $a=0$.
Assume that $a>0$: let $\eps>0$ be the
constant in Lemma \ref{l.barrier} and notice that,
since $u_k$ converges to $u_\infty= -a x_{n+1}$ uniformly on $\bar B_1^+$,
it must be $u_k\vert_{\de B_1} \leq w_\eps\vert_{\de B_1}$ definitively, 
where $w_\eps$ is the solution to the thin obstacle problem with 
boundary value $g_\eps(x) = -a|x_{n+1}| +\eps$.
By the comparison principle of Proposition~\ref{p.comparison}
$u_k\vert_{B_1} \leq w_\eps\vert_{B_1}$ for $k$
sufficiently large,
which in turn by Lemma \ref{l.barrier} leads to $u_k\vert_{B_{\sfrac34}'} \equiv 0$.
This is a contradiction to $0 \in \Gamma(u_k)$, thus establishing
that $a=0$.

Finally, since we have shown that any
convergent subsequence of $u_k$ is uniformly converging to $0$,
we conclude that the whole sequence $u_k$ converges uniformly
to $0$ on $\bar B_1$.
\end{proof}

\begin{proof}[Proof of Proposition \ref{p.C1}]
By Frehse's Theorem~\ref{t.frehse},
we need only to prove that the normal derivative $\de_{n+1} u$ 
is a continuous function in $B_1^+\cup B_1'$.
Moreover, since $\de_{n+1}u$ is analytic in $B_1^+\cup B_1'
\setminus \Gamma(u)$,
we have only to check its continuity at points of the free
boundary $\Gamma(u)\subseteq B_1'$.

Without loss of generality, we can assume that
$\underline{0}\in\Gamma(u)$ and we begin with showing that $u$ 
is differentiable at $\underline{0}$ with zero normal derivative:
\begin{equation}\label{e.derivata normale zero}
\lim_{t\to 0^+}\frac{u(0,t)}{t} = 0.
\end{equation}
We apply Lemma \ref{l.blowup} to any sequence 
$(t_k)_{k\in \N}$ with $t_k\downarrow 0$ and $z_k=\underline{0}$
for all $k$: the functions 
$u_k (x) = t_k^{-1}u(t_k x)$ converge uniformly to $0$ 
in $\bar B_1$. In particular,
\[
\lim_{k\to \infty}\frac{u(0,t_k)}{t_k} =
\lim_{k\to \infty} u_k(e_{n+1}) = 0.
\]
From the arbitrariness 
of the sequence $(t_k)_{k\in \N}$,
\eqref{e.derivata normale zero} in turn follows.

Next we prove the $\de_{n+1} u$ is continuous in $\underline{0}\in\Gamma(u)$.
Let $y_k \in B_1^+\cup \big(B_1'\setminus \Gamma(u)\big)$ be
a sequence of points converging to $\underline{0}$.
Let $t_k := \dist\big(y_k,\Gamma(u)\big)=|y_k-z_k|\to 0$,
with $z_k \in \Gamma(u)$.
Therefore $B_{t_k}(y_k)\cap \Gamma(u) =\emptyset$, and 
either $B_{t_k}(y_k)\cap \Lambda(u) =\emptyset$,
in which case we set $v(x) := u(x)$ for all $x\in B_{t_k}(y_k)$,
or $B_{t_k}(y_k)\cap B_1'\subseteq \Lambda(u)$ and we set 
\[
v(x) :=
\begin{cases}
u(x) & \text{if $x_{n+1}\geq 0$,}\\
-u(x) & \text{if $x_{n+1}<0$.}
\end{cases}
\]
In both cases $v$ is a solution to the minimal surface equation
in $B_{t_k}(y_k)$ 
(indeed, $u$ solves the minimal surface equation
in $B_{t_k}^+(y_k)$ either with null Neumann or with null Dirichlet 
boundary conditions on $B_{t_k}(y_k)\cap B_1'$, respectively;
therefore $v$ is readily regognized to be a solution in both cases).
Set $\tau_k:= 2\,|y_k-z_k|$ and let $v_k: B_1 \to \R$ 
be given by 
\[
v_k(x):=\frac{v(z_k+\tau_k\,x)}{\tau_k}.
\]
By Lemma \ref{l.blowup}, $v_k$ is uniformly converging to $0$.
Moreover, by possibly passing to a further subsequence, 
we can assume that $p_k:=\frac{y_k-z_k}{\tau_k} \to p\in
\de B_{\sfrac12}$.
Since, the functions $v_k$
are solutions of the minimal surface equation in $B_{\sfrac12}(p)$
and they are converging uniformly to $0$,
the regularity theory for the minimal surface equation
implies that the convergence is in fact smooth.
In particular, in both cases discussed above we get
\[
\lim_{k\to \infty}\de_{n+1}v(y_k) = 
\lim_{k\to \infty}\de_{n+1} v_k(p_k) = 0,
\]
thus concluding the continuity of $\de_{n+1}u$ at $\underline{0}$.
\end{proof}

%
%
\section{$C^{1,\alpha}$ regularity}\label{s:C1alpha}

This section is devoted to show the one-sided 
$C^{1,\alpha}(B_1^+\cup B_1')$ regularity.
To this aim, we need to consider approximate solutions
produced by the method of penalization.

\subsection{The penalized problem}
Let $g\in C^{2}(\R^{n+1})$ be a fixed boundary value for 
\eqref{e:ob prob} and let $u\in W^{1,\infty}(B_1)$ be
the unique solution to the thin obstacle problem.
For the rest of the section, we set $L:=\Lip(u)$.

We start off considering the following penalized problem:
let $\beta,\chi\in C^\infty(\R)$ be such that
\[
|t|-1\leq |\beta(t)|\leq |t|\quad\forall\;t\leq 0,\quad
\beta(t) = 0 \quad\forall\;t\geq 0,\quad
\beta'(t)\geq 0 \quad\forall\;t\in\R\,,
\]
\[
\chi(t) =
\begin{cases}
0 & \text{for } \; t\leq L,\\
\frac{1}{2}\,(t-2L)^2& \text{for } \; t > 3L,
\end{cases}
\quad \chi''(t) \geq 0\quad\forall\;t\in \R.
\]
For every $\eps>0$ set $\beta_\eps(t) := \eps^{-1}
\beta(\sfrac{t}\eps)$
and we introduce the energy 
\begin{equation*}
\mathscr{E}_\eps(v) := \int_{B_1} \left(\jac{v} +
\chi(|\nabla v|)\right) \, \d x 
+ \int_{B_1'}F_\eps(v(x',0))\,\d x',
\end{equation*}
where $F_\eps(t) :=2\int^t_0 \beta_\eps(s) \,\d s$.
Since the energy $\mathscr{E}_\eps$ is strictly 
convex and quadratic, there exists a unique minimizer $u_\eps\in
g+W^{1,2}_0(B_1)$.
Moreover, from the symmetry of $g$, it follows that
$u_\eps$ is also even symmetric with respect to $x_{n+1}$.

The Euler--Lagrange equation satisfied by $u_\eps$ is then
given by
\begin{equation}\label{e.beta eq}
\int_{B_1^+} 
A(\nabla u_\eps) \cdot \nabla \eta
\, \d x + 
\int_{B_1'}\beta_{\eps}(u_\eps)\;\eta\;\d x' = 0
\qquad \forall\; \eta \in H^1_0(B_1),
\end{equation}
with $A:\R^{n+1} \to \R^{n+1}$ being the vector field
\[
A(p):=\left((1+|p|^2)^{-\sfrac12}+
\chi'(|p|)\,|p|^{-1}\right)\,p.
\]
Note that for $|p|\leq L$ the second addend is actually null.
 
The following lemma establish the connection between the solutions
of the penalized problems and the solution to the thin obstacle problem.

\begin{lemma}\label{l.convergenza a sol}
Let $g\in C^2(\R^{n+1})$ be even symmetric with respect
to $x_{n+1}$ and $g\vert_{\R^n\times\{0\}}\geq 0$.
Then, the minimizers $u_\eps$ of $\mathscr{E}_\eps$ on $g+W^{1,2}_0(B_1)$
converge weakly in $W^{1,2}$ as $\eps$ goes to $0$ to the
solution $u$ to the thin obstacle problem \eqref{e:ob prob}.
\end{lemma}
\begin{proof}
From the definition of $\chi$ one readily verifies that
there exists a constant $C>0$ such that
$t^2\leq C(1 + \chi(t))$ for every $t\geq0$.
Thus, it follows that the approximate solutions $u_\eps$
have equi-bounded Dirichlet energy:
\begin{align*}
\int_{B_1} |\nabla u_\eps|^2\, dx &\leq 
C\cL^{n+1}(B_1) + C\int_{B_1} \chi(|\nabla u_\eps|)\, dx
\leq C\cL^{n+1}(B_1) + C\mathscr{E}_\eps(u_\eps)\\
& \leq C\cL^{n+1}(B_1) + C\mathscr{E}_\eps(u)
= C\cL^{n+1}(B_1) + C\int_{B_1}\jac{u}\,dx.
\end{align*}
Then, up to extracting a subsequence (not relabeled), 
there exists a function $u_0\in g + W^{1,2}_0(B_1)$ 
such that $u_\eps$ converges  to $u_0$ in $L^{2}(B_1)$
and the trace $u_\eps\vert_{B_1'}$ converges 
to $u_0\vert_{B_1'}$ in $L^{2}(B_1')$.

We next show that $u_0\vert_{B_1'}\geq 0$.
Recalling that $F_\eps$ is positive and monotone decreasing, 
we have by Chebyshev inequality
\[
F_\eps(-\delta)\,\cL^n\big(\{u_\eps< -\delta \}
\cap B_1'\big) \leq 
\int_{B_1'}F_\eps(u_\eps)\,\d x \leq \mathscr{E}_\eps(u_\eps) 
\leq\int_{B_1}\sqrt{1+|\nabla u|^2}\,\d x.
\]
Since $F_\eps(t)\uparrow\infty$ as $\eps\downarrow 0$
for all $t<0$ and $u_\eps\vert_{B_1'}\to u_0\vert_{B_1'}$
in $L^{2}(B_1')$, we conclude that
\[
\cL^n\big(\{u_0< -\delta \}\cap B_1'\big) = 0 \quad \forall \;\delta>0,
\]
which implies $u_0\vert_{B_1'}\geq 0$, 
{\ie}~$u_0\in\mathcal{B}_g$ where 
\[
\mathcal{B}_g := \Big\{w \in g+W_0^{1,2}(B_1) :
w\vert_{B_1'}\geq 0 \Big\}.
\]
Furthermore, $u_0$ is the unique minimizer in $\mathcal{B}_g$ 
of the energy $\mathscr{F}:\,W^{1,2}(B_1)\to[0,\infty)$
defined by
\[
\mathscr{F}(w):= \int_{B_1}\left(\jac{w} +
\chi(|\nabla w|)\right) \, \d x\,.
\]
Indeed, by convexity of $\mathscr{F}$, for every 
$w \in \mathcal{B}_g$ we have that
\begin{align*}
\mathscr{F}(u_0) & \leq
\liminf_{\eps\to0^+}
\mathscr{F}(u_\eps) \leq 
\liminf_{\eps\to0^+} \mathscr{E}_{\eps}(u_{\eps})
\leq
\liminf_{\eps\to0^+} \mathscr{E}_{\eps}(w)
= \mathscr{F}(w),
\end{align*}
since $\mathcal{B}_g\subset g + W_0^{1,2}(B_1)$
and $F_\eps(w) = 0$ for all $w\in \mathcal{B}_g$.
To conclude, we only need to notice that the unique
minimizer of $\mathscr{F}$ on $\mathcal{B}_g$ is exactly
the solution to the thin obstacle
problem $u$.
Indeed, $\mathcal{A}_g\subseteq \mathcal{B}_g$
and for every $w\in \mathcal{B}_g$ we have that
\begin{align*}
\mathscr{F}(u)&=\int_{B_1}\jac{u}\,dx \leq 
\int_{B_1}\jac{w}\,dx \leq \mathscr{F}(w),
\end{align*}
where we used that $\chi(|\nabla u|)\equiv 0$
and that $u$ is a minimizer of the thin obstacle problem
for the area functional among all competitors in 
$\mathcal{B}_g$, and not only in $\mathcal{A}_g$ 
(this follows from an approximation argument).

Finally, being the solution to the Signorini problem unique,
by Urysohn property 
we conclude that the whole family $(u_\eps)_{\eps>0}$
converges to $u$.
\end{proof}

\subsection{$W^{2,2}$ estimate}
Next we show that the solution to the penalized problem, 
as well as the solution 
to the thin obstacle problem, possess second derivatives in $L^2(B_1^+)$.
The proof is at all analogous to the standard $L^2$-theory
for quasilinear equations: we report it for readers convenience.

We recall the standard notation of the difference quotient 
\[
\tau_{h,i}f(x) := h^{-1} \big(f(x+he_i) - f(x) \big),
\]
if $x\in\{y\in B_1:\, y+he_i\in B_1\}$ and 
$\tau_{h,i}f(x):=0$ otherwise, 
where $f: B_1 \to \R$ is any measurable function and $e_i$ a coordinate 
vector, $i\in\{1,\ldots, n+1\}$. 
\begin{proposition}\label{p.H2}
The solutions $u_\eps$ to the penalized 
problems \eqref{e.beta eq}  for every $\eps>0$
and the solution $u$ to the thin obstacle problem satisfy the following property:
there exists a constant $C=C(n,L)>0$ such that, 
if either $v= u_\eps$ or $v= u$, then 
\begin{gather}\label{e.H2}
\int_{B_r^+(x_0)} |\nabla^2 v|^2\,\d x \leq \frac{C}{r^2}
\int_{B_{2r}^+(x_0)} |\nabla' v|^2\,\d x
\quad\forall\; x_0\in B_1^+\cup B_1',\;
\forall \; 0<r<\textstyle{\frac{1-|x_0|}{2}}\,.
\end{gather} 
\end{proposition}
\begin{proof}
The result is classical if $x_0\in B_1^+$ and
$B_r(x_0)\subset\subset B_1^+$. 
We shall prove only the case in which $x_0\in B_1'$,
and the general case follows by a covering argument.
Without loss of generality we may assume $x_0=\underline{0}$. 

We provide first an estimate for the horizontal derivatives of the weak gradient 
of $u_\eps$. Let $\zeta \in C_c^1(B_{2r})$, $2r<1$, be a test function with 
$\zeta\equiv 1$ in $B_r$ and $|\nabla \zeta|\leq C\,r^{-1}$
for some dimensional constant $C>0$.
We test \eqref{e.beta eq} with 
$\eta := \tau_{-h, i} \big(\zeta^2 \,\tau_{h,i} u_\eps \big)$, 
with $|h|<1-2r$ and $i\in\{1,\ldots, n\}$.
For convenience, in the following computation we omit to write the index $i\in\{1,\ldots, n\}$ 
in the notation of the difference quotients. We start off noticing that the first addend in \eqref{e.beta eq} 
rewrites as 
\begin{align}\label{e: test}
\int_{B_1^+} A(\nabla u_\eps) \cdot \nabla \eta \,\d x &=
-\int_{B_1^+} \tau_h\big(A(\nabla u_\eps)\big)
\cdot \nabla \big(\zeta^2 \,\tau_{h} u_\eps \big)\,\d x,
\end{align}
where we used the basic integration by parts formula for discrete derivatives
\[
\int (\tau_hf)\, \varphi\,\d x 
=- \int f\,(\tau_{-h}\varphi)\,\d x\quad 
\text{ $\forall f, \varphi$ measurable, $\varphi$ having compact support}\,.
\]
We now compute as follows: set
\[
\psi(t) := A\big((1-t)\nabla u_\eps(x)+t 
\nabla u_\eps(x+h e_i)\big);
\]
then,
\begin{align*}
\tau_h\big(A(\nabla u_\eps)\big)
& =\frac 1h \int_0^1\psi'(t)\,\d t
= \int_0^1 \nabla {A}\big( (1-t)\nabla u_\eps(x)+t 
\nabla u_\eps(x+h e_i)\big)\,\d t \, \tau_h(\nabla u_\eps)\\
& = : \mathbb{A}_\eps^{h}(x) \tau_h(\nabla u_\eps).
\end{align*}
Note that there exist constants $0<\lambda <\Lambda$
(depending on $L=\Lip(u)$)
such that 
\[
\lambda\,\Id_{n+1}\leq 
\mathbb{A}_\eps^{h}(x)
\leq \Lambda\,\Id_{n+1}
\quad\forall\; x\in B_1^+,
\]
because 
\begin{align*}
\nabla {A}(p) &=
\nabla\left(\frac{p}{\sqrt{1+|p|^2}} +
\chi'(|p|)\,\frac{p}{|p|}\right)\\
&=
\frac{\Id_{n+1}}{(1+|p|^2)^{\sfrac32}}+
\big((1+|p|^2)^{-\sfrac32} + \chi'(|p|)\,|p|^{-3}\big)
\big(|p|^2\, \Id_{n+1} - p\otimes p\big)
+\chi''(|p|)\, \frac{p\otimes p}{|p|^2}
\end{align*}
is uniformly elliptic and bounded.
Therefore, we can rewrite \eqref{e: test} as 
\begin{align*}
\int_{B_1^+} A(\nabla u_\eps) \cdot \nabla \eta \,\d x
& =-\int_{B_1^+} \mathbb{A}_\eps^{h}\,\tau_h(\nabla u_\eps)
\cdot \nabla \big( \zeta^2\,\tau_{h} u_\eps \big)\,\d x\\
& =-\int_{B_1^+} \Big(\zeta^2\,\mathbb{A}_\eps^{h}\,\tau_h(\nabla u_\eps)
\cdot \tau_h(\nabla u_\eps)\,
+2\zeta\,(\tau_{h} u_\eps) \,
\mathbb{A}_\eps^{h}\,\tau_h(\nabla u_\eps)
\cdot \nabla \zeta\Big)\,\d x.
\end{align*}
On the other hand, by the monotonicity of $\beta_\eps$ the second addend
in \eqref{e.beta eq} is non-positive. Indeed, being $\beta_\eps$ increasing, we have
\begin{align*}
&\int_{B_1'} \beta_\eps(u_\eps) \tau_{-h}\big(\zeta^2\,\tau_{h} u_\eps\big)\,\d x' =
-\int_{B_1'} \tau_{h}\big(\beta_\eps(u_\eps)\big) (\tau_{h} u_\eps)\, 
\zeta^2\,\d x'\\
&=-\int_{B_1'} \frac{\beta_\eps(u_\eps(x'+he_i))-\beta_\eps(u_\eps(x'))}{h} 
\,\frac{u_\eps(x'+he_i) - u_\eps(x')}{h} \,\zeta^2\,\d x'\leq 0.
\end{align*}
Thus, from \eqref{e.beta eq} we infer that 
\begin{align*}
\int_{B_1^+} \Big(\mathbb{A}_\eps^{h}\,\tau_h(\nabla u_\eps)
\cdot \tau_h(\nabla u_\eps)\, \zeta^2
+2\zeta\,\tau_{h}(u_\eps)\,
\mathbb{A}_\eps^{h}\,\tau_h(\nabla u_\eps)
\cdot \nabla \zeta\Big)\,\d x\leq 0.
\end{align*}
Hence, in view of Cauchy-Schwarz inequality and of the ellipticity of $\mathbb{A}_\eps^{h}$
we conclude that
\begin{align*}
\int_{B_1^+} |\tau_h (\nabla u_\eps)|^2\, \zeta^2\,\d x
\leq 4\frac{\Lambda}{\lambda}\int_{B_1^+} 
|\tau_{h} u_\eps|^2\,|\nabla \zeta|^2\,\d x.
\end{align*}
The latter estimate implies that $\nabla u_\eps$
has weak $i$-th derivative in $L^2(B_r^+)$,
for all $i\in\{1,\dots,n\}$, $r<\sfrac12$, with 
\begin{equation}\label{e.stima finale W22}
\int_{B_r^+}|\de_i(\nabla u_\eps)|^2\,\d x 
\leq \frac{C}{r^2} \int_{B_{2r}^+} |\partial_i u_\eps|^2\,\d x,
\end{equation}
for a constant $C>0$ depending only on $L$. 

To conclude the proof for $v=u_\eps$ it suffices to prove that $\de_{n+1}u_\eps$ 
has $(n+1)$-th weak derivative in $B_1^+$. 
Writing $A(p)= (A^1(p),\ldots, A^{n+1}(p))$, 
we have that
\begin{align*}
\de_jA^i(\nabla u_\eps) \de_{ij} u_\eps =0.
\end{align*}
Moreover, $\lambda\leq \de_{n+1}A^{n+1}(p) \leq \Lambda$ for every
$p\in \R^{n+1}$, from which we deduce that
\begin{equation}\label{e.derivata normale}
\de_{n+1}^2u_\eps =
\frac{1}{\de_{n+1}A^{n+1}(\nabla u_\eps)}
\sum_{(i,j)\neq(n+1,n+1)} \de_jA^i(\nabla u_\eps) \de^2_{i,j} u_\eps
\in L^2_\loc(B_1^+)\,.
\end{equation}
Hence, from \eqref{e.stima finale W22} 
and the fact that $\nabla A$ is bounded,
we get the estimate 
 \begin{equation}\label{e.stima n+1 W22-2}
 \int_{B_r^+}|\de_{n+1}(\nabla u_\eps)|^2\,\d x 
 \leq C\sum_{i=1}^n 
 \int_{B_r^+}|\nabla (\de_i u_\eps)|^2\,\d x \leq 
 \frac{C}{r^2} \int_{B_{2r}^+} |\nabla' u_\eps|^2\,\d x\,,
 \end{equation}
with $C=C(n,L)>0$.
Being estimates \eqref{e.stima finale W22} and \eqref{e.stima n+1 W22-2} uniform in $\eps$, 
in view of Lemma~\ref{l.convergenza a sol}, we can pass to the limit as $\eps\downarrow 0$ 
and infer that the same estimates hold for $u$ as well.
\end{proof}

\subsection{$C^{1,\alpha}$ estimate}
Next we prove that the minimizer $u$ of the Signorini problem
has weak derivatives in suitable De Giorgi classes on the flat 
part of the boundary. Here, we do follow the approach by Ural'tesva \cite{Ural86}
in conjunction with the one-sided continuity of the derivatives shown
in Proposition~\ref{p.C1}. In particular, the latter result 
is instrumental to establish the ensuing estimate \eqref{e.DG} for $\pm\de_{n+1}u$.

\begin{proposition}\label{p.C1alfa}
Let $u$ be the solution to the thin obstacle problem, 
then for some constant $C=C(n,L)>0$ the
function $v=\pm \partial_i u$, $i\in\{1,\ldots, n+1\}$,
satisfies for all $k\geq 0$
\begin{equation}\label{e.DG}
\int_{B_r^+(x_0) \cap \{v >k\}}|\nabla v|^2\,\d x
\leq \frac{C}{r^2} \int_{B_{2r}^+(x_0)}(v- k)_+^2\,\d x
\quad \forall\; x_0\in B_1', \;0<r<\textstyle{\frac{1-|x_0|}2}.
\end{equation}
\end{proposition}

\begin{proof}
We start off writing the equation satisfied by the horizontal derivatives
of the solution to the penalized problem \eqref{e.beta eq} and by testing it 
with $\eta=\de_i \zeta$, $i\in\{1,\ldots, n\}$, for $\zeta\in W^{2,2}(B_1)$ 
even symmetric with respect to $x_{n+1}$ and
$\textrm{spt}\zeta\cap\partial B_1=\emptyset$:
\begin{align}\label{e.test C1alfa}
0&=\int_{B_1^+} \de_i\left(
A(\nabla u_\eps)\right)
\cdot \nabla \zeta\,\d x+ \int_{B_1'}\de_i[\beta_\eps(u_\eps)]\,\zeta\,\d x'\notag\\
& = 
\int_{B_1^+} 
\nabla A(\nabla u_\eps)\, \nabla (\de_i u_\eps) \cdot \nabla \zeta\,\d x
+ \int_{B_1'}\beta_\eps'(u_\eps) \de_iu_\eps\,\zeta\,\d x'.
\end{align}
Note that \eqref{e.test C1alfa} makes sense as soon as 
$\zeta\in W^{1,2}(B_1^+)$ with  $\spt\zeta\cap(\partial B_1)^+=\emptyset$, 
thanks to the integrability estimates
in Proposition \ref{p.H2}.
Therefore, as $u_\eps\in W^{2,2}(B_1^+)$ we can choose 
$\zeta_\eps:= (\de_i u_\eps - k)_+\, \phi^2$ for $k\geq 0$ and having 
fixed $\phi\in C^1_c(B_1)$, because
$\zeta_\eps\in W^{1,2}(B_1^+)$ with $\spt\zeta_\eps\cap(\partial B_1)^+=\emptyset$.
With this choice at hand, note then that 
\begin{align}\label{e.beta'}
\int_{B_1'}\beta_\eps'(u_\eps) \de_iu_\eps\,\zeta_\eps\,\d x'
& = \int_{B_1'}\beta_\eps'(u_\eps)
\de_iu_\eps\,(\de_i u_\eps - k)_+\, \phi^2\,\d x'
\geq 0.
\end{align}
For what concerns the remaining terms, we recall that
$\nabla \zeta_\eps = \phi^2\nabla (\de_i u_\eps)
\chi_{\{\de_i u_\eps >k\}}
+ 2\phi(\de_i u_\eps - k)_+\nabla \phi$.
Therefore, we have that
\begin{align*}
0&\geq 
\int_{B_1^+\cap \{\de_i u_\eps >k\}} 
\phi^2 \nabla A(\nabla u_\eps)\, \nabla (\de_i u_\eps) \cdot
\nabla (\de_i u_\eps)\, \d x\\
&
\quad+\int_{B_1^+} 2\phi(\de_i u_\eps - k)_+
\nabla A(\nabla u_\eps)\, \nabla (\de_i u_\eps) \cdot \nabla \phi
\,\d x.
\end{align*}
Then, a standard argument implies
\begin{align*}
\int_{B_1^+ \cap \{\de_i u_\eps >k\}}&
\phi^2|\nabla (\de_iu_\eps)|^2\,\d x  \leq 
4\frac{\Lambda}{\lambda}
\int_{B_1^+}(\de_i u_\eps - k)_+^2
\,|\nabla \phi|^2\,\d x
\end{align*}
In particular, for every $k\geq 0$ 
and for every $x_0\in B_1'$ and 
$0<2r<1-|x_0|$
if $\phi\in C^1_c(B_{2r}(x_0))$ and $\phi\equiv 1$ 
on $B_r(x_0)$ with $|\nabla\phi|\leq \sfrac{C}r$
\begin{equation}\label{e.DG eps}
\int_{B_r^+ \cap \{\de_i u_\eps >k\}}
|\nabla (\de_iu_\eps)|^2\,\d x
\leq \frac{C}{r^2} \int_{B_{2r}^+}(\de_i u_\eps - k)_+^2
\,\d x\,,
\end{equation}
for some $C=C(L)>0$. In exactly the same way, 
by testing \eqref{e.beta eq} with 
$\zeta_\eps:= (-\de_i u_\eps - k)_+\, \eta^2$, we derive the analogous estimate
\begin{equation}\label{e.DG- eps}
\int_{B_r^+ \cap \{\de_i u_\eps <-k\}}|\nabla (\de_iu_\eps)|^2\,\d x
\leq \frac{C}{r^2} \int_{B_{2r}^+}(-\de_i u_\eps - k)_+^2\,\d x\,,
\end{equation}
for all $k\geq 0$ and $i\in\{1,\ldots, n\}$.
Estimate \eqref{e.DG} for $\pm\partial_i u$, with $i=1,\dots, n$, follows at once 
by passing to the limit as $\eps\downarrow 0$ in \eqref{e.DG eps} and \eqref{e.DG- eps}, 
respectively.

For what concerns the partial derivative in direction $n+1$,
we test the equation
\eqref{e.test C1alfa} with $\eta=\de_{n+1}\zeta$,
for $\zeta\in W^{2,2}(B_1^+)$ 
with $\textrm{spt}\zeta\cap(\partial B_1)^+=\emptyset$:
\begin{align}\label{e.test C1alfa n+1}
0=&\int_{B_1^+}\de_{n+1}
\left(A(\nabla u_\eps)\right)\cdot \nabla \zeta\,\d x
+\int_{B_1'} A(\nabla u_\eps)\cdot \nabla \zeta\,\d x'
- \int_{B_1'}\beta_\eps(u_\eps)\,\de_{n+1}\zeta\,\d x'\notag\\
=& 
\int_{B_1^+}\de_{n+1}
\left(A(\nabla u_\eps)\right)\cdot \nabla \zeta\,\d x
+\int_{B_1'} A'(\nabla u_\eps)\cdot \nabla' \zeta\,\d x',
\end{align}
where we set $A'(p) := (A^1 (p), \ldots, A^n(p))$.
The last equality holds thanks to Euler-Lagrange 
condition induced by \eqref{e.beta eq}:
\begin{equation}\label{e.beta prob}
\begin{cases}
\div\big(A(\nabla u_\eps)\big) = 0
& \textup{in }\;B_1^+,\\
A^{n+1}(\nabla u_\eps)
= \beta_\eps\big(u_\eps\big)
& \textup{on }\;B_1'.
\end{cases}
\end{equation}
For $0<k\leq\|\de_{n+1}u \|_{L^\infty(B_1^+)}$ set
\[
\zeta_{\delta} := \phi^2 \gamma_\delta(-\de_{n+1}u- k),
\]
where $\delta>0$ will be suitably chosen,
$\gamma_\delta\in C^\infty(\R)$ is an increasing function 
such that $\gamma_\delta(t)=0$ for $t\leq 0$, $\gamma_\delta(t)>0$ for $t>0$,
$\gamma_\delta(t)=t-\delta$ for $t\geq 2\delta$, 
$|\gamma_\delta'(t)|\leq1$ (such a function can be easily exhibited),
and 
$\phi\in C^\infty_c(B_{2r}(x_0))$, $\phi|_{B_r(x_0)}=1$,
$|\nabla\phi|\leq \sfrac{C}r$.
We use 
$\de_{n+1}u\in C^0(B_1^+\cup B_1')$ (cf. Proposition~\ref{p.C1})   
to infer that for $k>0$ 
the set $B_1'\cap \{\de_{n+1}u < -k\}$ is an
open set with compact closure in $\Lambda(u)$
(recall that $\partial_{n+1}u=0$ on $B_1'\setminus\Lambda(u)$).
This implies that, if $\delta>0$ is sufficiently small,
$\zeta_\delta\in C^\infty(B_1^+)$ with 
$\textrm{spt}\zeta_\delta\cap(\partial B_1)^+=\emptyset$.
Indeed, $u\in C^\infty(B_{r}^+(y_0))$
for all $y_0\in  B_1'\cap\spt\zeta_\delta$
and $r<\dist\big(B_1'\cap \{\de_{n+1}u \leq -k\}, 
B_1'\cap \{\de_{n+1}u =0\}\big)$, 
being $u$ itself minimum of the area
problem with null Dirichlet boundary conditions on 
$B_{r}'(y_0)$.
Taking $\zeta=\zeta_\delta$ we evaluate each addend in 
\eqref{e.test C1alfa n+1} separately. To begin with, 
the first term rewrites as
\begin{align*}
I^{\eps,\delta}:=& \int_{B_1^+}
\de_{n+1}\left(A(\nabla u_\eps)\right)\cdot \nabla \zeta_\delta\,\d x
= 2\int_{B_1^+} \phi\, \gamma_\delta (-\de_{n+1}u-k)
\nabla A(\nabla u_\eps) \nabla(\de_{n+1}u_\eps)
\cdot \nabla \phi\, \d x\\
&\quad - \int_{B_1^+} \phi^2\gamma_\delta'(-\de_{n+1}u-k) 
\nabla A(\nabla u_\eps) \nabla(\de_{n+1}u_\eps)
\cdot \nabla (\de_{n+1}u)\,\d x.
\end{align*}
Taking the limits as $\eps\downarrow 0$ in each term above,
since $\nabla A$ is a Lipschitz function and $\nabla u_\eps \to
\nabla u$ in $L^2$ and $\nabla(\de_{n+1}u_\eps)\weak
\nabla(\de_{n+1}u)$ in $L^2$, we conclude that 
\begin{align*}
\lim_{\eps\to 0} I^{\eps, \delta} &=
2\int_{B_1^+} \phi\, \gamma_\delta (-\de_{n+1}u-k)
\nabla A(\nabla u) \nabla(\de_{n+1}u)
\cdot \nabla \phi\, \d x\\
&\quad - \int_{B_1^+} \phi^2\gamma_\delta'(-\de_{n+1}u-k) 
\nabla A(\nabla u) \nabla(\de_{n+1}u)
\cdot \nabla (\de_{n+1}u)\,\d x.
\end{align*}
Moreover, since $\gamma_{\delta}(-\de_{n+1} u -k )
\to(-\de_{n+1} u - k)_+$ strongly in $W^{1,2}(B_1^+)$
as $\delta \downarrow 0$, then we infer
\begin{align*}
\lim_{\delta\downarrow 0}\lim_{\eps\downarrow 0}I^{\eps,\delta}=&
2\int_{B_1^+} \phi\, (-\de_{n+1}u-k)_+
\nabla A(\nabla u) \nabla(\de_{n+1}u)
\cdot \nabla \phi\, \d x\\
&\quad - \int_{B_1^+\cap \{\de_{n+1}u\leq -k\}}
\phi^2 \nabla A(\nabla u) \nabla(\de_{n+1}u)
\cdot \nabla (\de_{n+1}u)\,\d x.
\end{align*}
Similarly, to deal with the second addend in 
\eqref{e.test C1alfa n+1} we argue as follows: 
as $\nabla u_\eps\to \nabla u$ strongly in 
$L^{2}_{\textup{loc}}(B_1')$ by Proposition~\ref{p.H2} 
and the compactness of the trace operator, 
the Lipschitz continuity of $A'$
implies for all $\delta>0$ that 
\begin{align*}
\lim_{\eps\downarrow 0} 
\int_{B_1'} A'(\nabla u_\eps)\cdot \nabla'\zeta_\delta\,\d x'
=\int_{B_1'} A'(\nabla u)\cdot \nabla'\zeta_\delta=0\,.
\end{align*}
In the last equality we have used that
$B_1'\cap\spt\zeta_\delta\subset\subset\Lambda(u)$,
and being (the trace of) $u$ in $W^{1,2}(B_1')$ by Proposition~\ref{p.H2}, 
then $\nabla'u=0$ $\cL^n$ a.e. on 
$B_1'\cap\spt\zeta_\delta$, so that 
$A'(\nabla u)=0$ $\cL^n$ a.e. on 
$B_1'\cap\spt\zeta_\delta$.

Hence, by using the ellipticity of $\nabla A$ we infer that
for every $k> 0$, by H\"older's inequality 
\begin{equation}\label{e.DG n+1}
\int_{B_r^+ \cap \{\de_{n+1} u <-k\}}
|\nabla (\de_{n+1}u)|^2\,\d x
\leq \frac{C}{r^2} \int_{B_{2r}^+}(-\de_{n+1} u - k)_+^2
\,\d x\,.
\end{equation}

Clearly, \eqref{e.DG n+1} holds for $k= 0$ by letting 
$k\downarrow 0$ in the inequality itself, and also for 
$k>\|\partial_{n+1}u\|_{L^\infty(B_1^+)}$ being 
trivial in those cases.
The case of $\de_{n+1}u$ is treated similarly.
\end{proof}

We are now ready to establish the claimed
one-sided $C^{1,\alpha}$ regularity of $u$:
the argument follows closely Ural'tesva \cite[Lemmata~2, 3]{Ural86}
and Giaquinta and Giusti \cite{GiaGiu75}.

\begin{corollary}
Let $u$ be the solution to the thin obstacle problem,
then $u\in C^{1,\alpha}_{loc}(B_1^+\cup B_1')$
for some $\alpha\in(0,1)$.
\end{corollary}
\begin{proof}
By standard results in elliptic regularity we have that 
$u\in C^{\infty}(B_1^+)$. 
Let $x_0\in B_1'$, $\rho\in(0,1-|x_0|)$
and $\rho_j:=2^{-j}\rho$, $j\geq 0$.
We start off considering the case
\begin{equation}\label{e:Lambdau dens}
\cL^n(\Lambda(u)\cap B_{\rho_j}'(x_0))\geq
\sfrac12\, \cL^n(B_{\rho_j}'(x_0)).
\end{equation}
Then for all $i\in\{1,\ldots, n\}$ we also get
\begin{equation*}
\cL^n(\{\partial_iu=0\}\cap B_{\rho_j}'(x_0))\geq \sfrac12\, 
\cL^n(B_{\rho_j}'(x_0)).
\end{equation*}
Let 
$i\in\{1,\ldots, n\}$ be fixed and
set $k_{j}:=\frac12(\max_{B_{\rho_{j}}'(x_0)} \de_i u+
\min_{B_{\rho_{j}}'(x_0)} \de_i u)$.
Without loss of generality, we can assume that $k_j\geq0$
(if this is not the case, we consider $-\de_iu$).
Then, 
\[
\cL^n\big(\{ \de_i u\leq k_{j}\}\cap B_{\rho_{j}}'(x_0)\big)
\geq \sfrac12\, \cL^n\big(B_{\rho_{j}}'(x_0)\big).
\]
By Proposition~\ref{p.H2}, a contradiction argument yields 
that the Poincar\'e type inequality
\[
\|( \de_i u-k)_+\|_{L^2(B_1^+)}\leq 
C\|\nabla( \de_i u-k)_+\|_{L^2(B_1^+)} \quad \forall\; k\geq k_j\geq 0,
\]
for some constant $C=C(n)>0$. Hence, by taking into account 
\eqref{e.DG} in Proposition~\ref{p.C1alfa},
the usual De Giorgi's argument can be run to conclude that 
\begin{equation}\label{e:osc decay i}
 \text{osc}_{B_{\rho_{j+1}}^+(x_0)}(\partial_iu)\leq
 \kappa\, \text{osc}_{B_{\rho_{j}}^+(x_0)}(\partial_iu)
\end{equation}
for all $i\in\{1,\ldots, n\}$, where $\kappa\in(0,1)$ depends only on $L$ 
(cf. \cite[Lemma~7.2]{Giusti03}).

On the other hand, if \eqref{e:Lambdau dens} does not hold, then by virtue 
of the (ambiguous) boundary conditions in \eqref{e.signorini bc} 
\[ 
\cL^n(\{\partial_{n+1}u=0\}\cap B_{\rho_{j}}'(x_0))\geq
\sfrac12 \cL^n(B_{\rho_{j}}'(x_0)).
\]
Note that $\partial_{n+1}u\vert_{B_1'}\leq 0$,
therefore
$k_{n+1}:=\frac12(\max_{B_{\rho_{j}}'(x_0)}(-\de_{n+1} u)+
\min_{B_{\rho_{j}}'(x_0)}(-\de_{n+1} u))\geq 0$.
Thus arguing as above, in view of \eqref{e.DG}
we conclude that 
\begin{equation}\label{e:osc decay n}
\text{osc}_{B_{\rho_{j+1}}^+(x_0)}(\partial_{n+1}u)
\leq \kappa\, \text{osc}_{B_{\rho_{j}}^+(x_0)}(\partial_{n+1}u)
\end{equation}
where $\kappa\in(0,1)$ depends only on $L$.

By means of estimates \eqref{e:osc decay i} and \eqref{e:osc decay n}, 
we next show that for some constant $C=C(L)>0$ and for all $r\in(0,1-|x_0|)$ 
\begin{equation}\label{e:decay osc}
\text{osc}_{B_{r}^+(x_0)}(v)\leq C\,r^\alpha\,,
\end{equation}
either for $v=\partial_iu$ for all $i\in\{1,\ldots,n\}$, 
or for $v=\partial_{n+1}u$.
With this aim, fix $N\in\N$ and consider the radii $\rho_j$ for 
$0\leq j\leq 2N-1$. Clearly, we can find (at least) $N$ radii 
$\rho_{j_h}$, $h=1,\ldots,N$, such that one between \eqref{e:osc decay i} and 
\eqref{e:osc decay n} holds for all such $h$'s. In particular, we infer 
that for all $1\leq h\leq N$ 
\[
\text{osc}_{B_{\rho_{j_{h+1}}}^+(x_0)}(v)\leq \kappa\, \text{osc}_{B_{\rho_{j_h}}^+(x_0)}(v)\,,
\]
with the function $v$ being equal either to $\partial_{n+1}u$ 
or to $\partial_iu$, in the latter case any $i\in\{1,\ldots, n\}$ works. 
Thus, iteratively, we conclude that 
\[
\text{osc}_{B_{\rho_{2N}}^+(x_0)}(v)\leq
\text{osc}_{B_{\rho_{j_{N+1}}}^+(x_0)}(v)\leq \kappa^{N+1}
\text{osc}_{B_{\rho}^+(x_0)}(v).
\]
Therefore, if $r\in(0,\rho)$ let $N\in\N$ be 
such that $r\in[\rho_{2N+1},\rho_{2N})$ 
we conclude then that
\[
 \text{osc}_{B_{r}^+(x_0)}(v)\leq
\text{osc}_{B_{\rho_{2N}}^+(x_0)}(v)\leq 
\big(\sfrac r\rho\big)^{\sfrac{|\log_2\kappa|}2}
\text{osc}_{B_{\rho}^+(x_0)}(v)=C\,r^\alpha.
\]

Actually, the last inequality always holds true for $\partial_{n+1}u$.
Indeed, considering the level $k=0\vee\min_{B_r^+(x_0)}v$ in 
Proposition~\ref{p.C1alfa}, with $v=\pm\partial_iu$ and $i\in\{1,\ldots, n\}$, 
from \eqref{e:decay osc} we infer that 
\[
\int_{B_r^+(x_0)}|\nabla v|^2\,\d x\leq C\,r^{n-1+2\alpha}.
\]
Hence, if \eqref{e:decay osc} holds for $v=\partial_iu$ for all $i\in\{1,\ldots, n\}$, 
then by using the estimate deriving from \eqref{e.stima n+1 W22-2} as 
$\eps\downarrow 0+$ and the latter inequality we conclude that
\[
\int_{B_r^+(x_0)}|\nabla(\partial_{n+1}u)|^2\,\d x\leq C\,r^{n-1+2\alpha}.
\]
In turn, Morrey's theorem implies that
\[
\text{osc}_{B_{r}^+(x_0)}(\partial_{n+1}u)\leq C\,r^\alpha.
\]
Hence, in any case we have shown that $\partial_{n+1}u
\in C^{0,\alpha}_{loc}(B_1')$.
In particular, we can infer that the co-normal derivative of $u$
is H\"older continuous in $B_1'$ in view of the boundary 
conditions in \eqref{e.signorini bc}:
\begin{equation*}
\frac{\de_{n+1}u}{\jac{u}} = 
\frac{\de_{n+1}u}{\sqrt{1+|\de_{n+1}u|^2}}
 \in C^{0,\alpha}_{loc}(B_1').
\end{equation*}
Note that the co-normal derivative is zero on $B_1'\setminus\Lambda(u)$.

We next use interior regularity and boundary regularity for the Dirichlet 
problem for the minimal surface equation together with an ad-hoc argument 
to infer that $u\in C^{1,\beta}_{\loc}(B_1^+\cup B_1')$
for some $\beta=\beta(n,L)\in(0,\alpha)$, recalling that $L=\Lip(u)$. 
For the sake of simplicity we show that $u\in C^{1,\beta}\big(B_{\sfrac34}^+ \cup B_{\sfrac34}'\big)$.
Let  $x_0\in B_{\sfrac 34}'$ and $r\in(0,\sfrac14)$.
If  $B_r'(x_0)\subseteq \Lambda(u)$,  we conclude by the regularity theory for the Dirichlet problem for uniformly elliptic equations (cf. \cite{GiaGiu75}) that for some $\beta>0$
\begin{equation}\label{e:decayubdry0}
\Phi(x_0,\rho):=
\int_{B_\rho^+(x_0)}|\nabla u-(\nabla u)_{B_\rho^+(x_0)}|^2dx
\leq C\Big(\frac \rho r\Big)^{n+2\beta}\Phi(x_0,r),
\end{equation}
provided that $\rho<r$, where $(v)_{E}:=\fint_Ev(x)\,\d x$ denotes the average of a function $v$ in the set $E$.
Instead, if there exists $z\in \Gamma(u) \cup (B_{r}'(x_0)\setminus \Lambda(u))$, then we show that
%
\begin{equation}\label{e:decayubdry}
\Phi(x_0,\rho)
\leq C\Big(\frac \rho r\Big)^{n+2\beta}\Phi(x_0,r)
+C[g]_{C^{0,\alpha}(B_{\sfrac 34})}^2r^{n+2\alpha},
\end{equation}
provided that $4\rho<r$.
Note that \eqref{e:decayubdry} and 
\cite[Lemma 7.3]{Giusti03} yield for all $\rho<r\leq\sfrac14$
\begin{equation}\label{e:decayphibdry}
\Phi(x_0,\rho)\leq C\Big(\frac 1{r^{n+2\beta}}
\Phi(x_0,r)+[g]_{C^{0,\alpha}(B_{\sfrac 34})}^2\Big)\rho^{n+2\beta},
\end{equation}
with $C=C(n,L,\alpha,\beta)>0$.

With the aim of proving \eqref{e:decayubdry}, let $w$ be the solution of 
\[
\begin{cases}
\div\Big(\frac{\nabla w}{\jac{w}}\Big)=0 & B_r^+(x_0),\\
\de_{n+1}w
= 0 & B_r'(x_0), \\
w=u & (\partial B_r(x_0))^+.
\end{cases}
\]
The existence of $w$ is guaranteed by an even reflection across 
$B_r'(x_0)$ of the boundary datum and by applying classical results on the 
existence of minimal surfaces with given Dirichlet boundary conditions 
(cf. \cite[Chapter 1]{Giusti03}).
By simple triangular inequalities, we have
\begin{align}\label{e:campanato} 
\Phi(x_0,r)
\leq 6\int_{B_r^+(x_0)}|\nabla u-\nabla w|^2dx
+4\int_{B_r^+(x_0)}|\nabla w-(\nabla w)_{B_r^+(x_0)}|^2dx.
\end{align}
We estimate the right hand side in \eqref{e:campanato} starting 
with the first addend. With this aim test \eqref{e.signorini bc} with 
$u-w\in H^1_0(B_r(x_0))$ to deduce that 
\[
\int_{B_r^+(x_0)}
\Big(\frac{\nabla u}{\jac{u}}-\frac{\nabla w}{\jac{w}}\Big)
\cdot\nabla(u-w)\d x+\int_{B_r'(x_0)}g(x')(u(x')-w(x'))\d x'=0
\]
where we have set $g:=\frac{\de_{n+1}u}{\sqrt{1+|\de_{n+1}u|^2}}$. 
Recall that $g\in C^{0,\alpha}_{\loc}(B_1')$ and $g(z)=0$.
In particular,
by the Divergence theorem we get 
\begin{align}\label{e:g holder}
(1+L^2&)^{-\sfrac32}\int_{B_r^+(x_0)}|\nabla(u-w)|^2\d x\leq
\int_{B_r'(x_0)}|g(x')-g(z)||u(x')-w(x')|\d x'\notag\\
&\leq [g]_{C^{0,\alpha}(B_{\sfrac 34}')}
(2r)^\alpha\int_{B_r'(x_0)}|u(x')-w(x')|\d x'\notag\\
&\leq [g]_{C^{0,\alpha}(B_{\sfrac 34})}(2r)^\alpha
\int_{B_r^+(x_0)}\div(|u-w|e_{n+1})\d x\notag\\
&\leq [g]_{C^{0,\alpha}(B_{\sfrac 34})}(2r)^\alpha
\int_{B_r^+(x_0)}|\nabla(u-w)|\d x\notag\\
&\leq 2^\alpha\omega_n^{\sfrac12}[g]_{C^{0,\alpha}(B_{\sfrac 34})}r^{\alpha+\sfrac n2}\|\nabla(u-w)\|_{L^2(B_r^+(x_0))}\,.
\end{align}
Hence, for some constant $C=C(n,L)>0$ we deduce that
\begin{equation}\label{e:nablaumenow}
\int_{B_r^+(x_0)}|\nabla(u-w)|^2\d x\leq 
C[g]_{C^{0,\alpha}(B_{\sfrac 34})}^2r^{n+2\alpha}.
\end{equation}
For the second term, we note that $w\in W^{2,2}(B_\rho^+(x_0))$
for every $\rho<r$ by arguing as in  Proposition~\ref{p.H2}. 
Moreover, if $i\in\{1,\ldots, n\}$ the function $\de_iw$
is a solution of 
\begin{equation}\label{e:pdedeiw}
\begin{cases}
\div(\mathbb{B}(x)\nabla(\de_iw))=0 & B_r^+(x_0),\\
\de_iw=0 & B_r'(x_0),
\end{cases}
\end{equation}
where the measurable matrix field $\mathbb{B}$ is given by
\[
\mathbb{B}(x):=\frac{\mathrm{Id}}{\jac{w}}
-\frac{\nabla w\otimes\nabla w}{(1+|\nabla w|^2)^{\sfrac32}}.
\]
By the a priori gradient estimate for the minimal surface equation
by Bombieri, De Giorgi and Miranda \cite{BoDGMi69}, there exists a constant $C>0$
such that
\[
\|Dw\|_{L^\infty(B_{\sfrac{r}2}(x_0))} \leq C e^{Cr^{-1}\|w\|_{L^\infty(B_{r}(x_0))}}
= C e^{Cr^{-1}\|u\|_{L^\infty(\partial B_{r}(x_0))}} \leq 
C e^{CL}.
\]
In particular, the matrix field $\mathbb{B}$
is coercive and bounded in $B_{\sfrac{r}2}$, with bounds depending
only on the Lipschitz constant $L$
of the solution $u$ to the thin obstacle problem.
Thus, by De~Giorgi's theorem \cite{DeG57} we have that 
$\de_iw\in C^{0,\beta}_{\loc}(B_r^+)$ with $\beta=\beta(n,L)$ and
\begin{equation}\label{e:dgw}
\int_{B_\rho^+(x_0)} |\de_i w-(\de_i w)_{B_\rho^+(x_0)}|^2\,\d x\leq C\Big(\frac{\rho}{r}\Big)^{n+2\beta}
 \int_{B_r^+(x_0)} |\de_i w-(\de_i w)_{B_r^+(x_0)}|^2\,\d x\,
 \quad \text{$\forall\;2\rho<r$.}
\end{equation}
In addition, being $\de_iw$ a solution of \eqref{e:pdedeiw}, it satisfies a Caccioppoli's inequality
\begin{equation}\label{e:caccioppoliw}
 \int_{B_{\rho}^+(x_0)} |\nabla(\de_i w)|^2\,\d x\leq
 \frac{C}{\rho^2}
 \int_{B_{2\rho}^+(x_0)} |\de_i w-(\de_i w)_{B_{2\rho}^+(x_0)}|^2\,\d x\,,
\end{equation}
with $4\rho<r$. 
Using the equation we can bound $\de_{n+1}^2w$ with the other derivatives (cf. \eqref{e.derivata normale} and \eqref{e.stima n+1 W22-2}) as follows
 \begin{equation}\label{e:nabladuew}
 \int_{B_\rho^+(x_0)}|\nabla(\de_{n+1}w)|^2\,\d x 
 \leq C \sum_{i=1}^n
 \int_{B_\rho^+(x_0)}|\nabla (\de_i w)|^2\,\d x\,,
 \end{equation}
with $C=C(n,L)>0$. 
Then, Poincare's inequality together with \eqref{e:dgw}, \eqref{e:caccioppoliw} and \eqref{e:nabladuew} give
\[
 \int_{B_\rho^+(x_0)} |\de_{n+1} w-(\de_{n+1} w)_{B_r^+}|^2\,\d x\leq  C\Big(\frac{\rho}{r}\Big)^{n+2\beta}\sum_{i=1}^n
 \int_{B_r^+(x_0)}|\de_i w-(\de_i w)_{B_r^+(x_0)}|^2\,\d x\,.
\]
Estimate \eqref{e:decayubdry} then follows at once from \eqref{e:campanato}, \eqref{e:nablaumenow}, \eqref{e:dgw}
and the latter inequality.

In addition, if $x_0\in B_1^+$ and $r\leq\dist(x_0,B_1')$, then
$B_r(x_0)\subset B_1^+$. 
Hence, by the standard regularity theory for uniformly elliptic equations  we have for all $\rho<r$
\begin{equation*}
\Phi(x_0,\rho)\leq C\Big(\frac{\rho}{r}\Big)^{n+2\beta}
\Phi(x_0,r),
\end{equation*}
$C=C(n,L)>0$ (cf. \cite{Giusti03}). 

From what we have proven, we deduce that  
there exists
$C=C(n,L,\alpha,\beta)>0$ such that 
for all $x_0\in B^+_{\sfrac 34}$ and $\rho<\sfrac 14$ 
\begin{equation*}
\Phi(x_0,\rho)\leq C\rho^{n+2\beta}\,,
\end{equation*}
from which the conclusion $u\in C^{1,\beta}\big(B^+_{\sfrac 34} \cup B_{\sfrac 34}'\big)$
readily follows by Campanato's theorem \cite{Giusti03}.
\end{proof}

%
%
\section{Optimal $C^{1,\sfrac12}$-regularity}\label{s:optimal}
In this section we deduce the optimal $C^{1,\sfrac12}$-regularity
of the solutions $u$ to the thin obstacle problem
from results by Simon and Wickramasekera \cite{SiWi16} on stationary
graphs of two-valued functions.
We give few preliminaries on the topic.
We consider pairs of real valued Lipschitz functions $U=\{u_1, u_2\}$
with the components $u_i$ defined on an open subset 
$\Omega\subset\R^N$. The union of the graphs of $u_1, u_2$,
namely
\[
\gr_U := \big\{(x,u_i(x)): x\in \Omega,\;i=1,2\big\}
\]
naturally inherits the structure of rectifiable varifold, 
which by a slight abuse of notation we keep denoting $\gr_U$.
Note that, $\gr_U= \gr_{u_1}+\gr_{u_2}$ as varifolds, where
$\gr_{u_i}$ denotes the varifolds associated to the graphs of the
real valued functions $u_i$.
Following \cite{SiWi16} we say that
$u$ is a two-valued minimal graph if $\gr_U$ is stationary
for the area functional, {\ie}
\[
\int_{\gr_U} \div_{\gr_U} Y \,\d\cH^{n+1} = 0
\quad \forall\; Y\in C^\infty_c(\Omega\times \R),
\]
where $\div_{\gr_U} Y$ denotes the tangential divergence of $Y$
in the direction of the tangent to $\gr_U$.
Clearly, if $u_1$ and $u_2$ are both solutions to the minimal
surface equation, then $u$ is a two-valued minimal graph,
but the vice-versa does not hold.
For more on multiple valued graphs we refer to 
\cite{DS-memo, DS-annali}.
In particular, we recall the definition of the metric for two-points:
$U=\{u_1, u_2\}$ and $V=\{v_1, v_2\}$,
\begin{align*}
\mathcal{G}(U, V) := \min \Big\{
\sqrt{|u_1 - v_1|^2 + |u_2 - v_2|^2},
\sqrt{|u_1 - v_2|^2 + |u_2 - v_1|^2}\Big\}.
\end{align*}
For two-valued functions the usual notion
of continuity and H\"older continuity can be accordingly introduced.
Moreover, a two-valued function $U$ is $C^1$
if there exists a continuous two-valued function
$DU = \{Du_1, Du_2\}$ with $Du_i\in \R^n$ such that,
setting $V_x(y) = \{u_1(x) + Du_1(x)(y-x),u_2(x) + Du_2(x)(y-x)\}$,
we have
\[
\lim_{y\to x}\frac{\mathcal{G}(U(y),V_x(y))}{|x-y|} =0.
\]
Finally, we say that $U$ is $C^{1,\alpha}$ if
$DU$ is H\"older continuous with exponent $\alpha$.

The link between the thin obstacle problem for the area functional
and the two-valued minimal graphs is given in the next proposition.

\begin{proposition}\label{p:ob vs 2-val}
Let $u$ be a solution to the thin
obstacle problem \eqref{e:ob prob}.
Then, the multiple-valued map $U= \{u, -u\}$
is a minimal two-valued graph.
\end{proposition}

\begin{proof}
According to the definition
of minimal two-valued graphs, we need to show that
\begin{equation}\label{e.stazionario}
\int_{\gr_U} \div_{\gr_U} Y \,\d\cH^{n+1} = 0
\quad \forall\; Y\in C^1_c(B_1\times \R).
\end{equation}
To this aim, we set
\[
\gr_1 = \gr_{u\vert_{\{x_{n+1}\geq 0 \}}},\quad
\gr_2 = \gr_{u\vert_{\{x_{n+1}\leq 0 \}}},\quad
\gr_3 = \gr_{-u\vert_{\{x_{n+1}\geq 0 \}}},\quad\text{and}\quad
\gr_4 = \gr_{-u\vert_{\{x_{n+1}\leq 0 \}}}.
\]
Clearly, we have that
\begin{equation}\label{e.somma}
\int_{\gr_U} \div_{\gr_U} Y \,\d\cH^{n+1} =
\sum_{i=1}^4\int_{\gr_i} \div_{\gr_i} Y \,\d\cH^{n+1}
\quad \forall\; Y\in C^\infty_c(B_1\times \R).
\end{equation}
Note that $u\vert_{\{x_{n+1}\geq 0 \}}$ and 
$u\vert_{\{x_{n+1}\leq 0 \}}$ are $C^1$ functions (cf. Proposition~\ref{p.C1}), 
therefore, $\gr_1$, $\gr_2$, $\gr_3$, $\gr_4$ are $C^1$-smooth submanifolds with 
boundary.
Let $\eta_i\in\R^{n+2}$ be the external
co-normal to $\de\gr_i$
({\ie} $|\eta_i|=1$, $\eta_i$ is normal to $\de \gr_i$
and tangent to $\gr_i$, pointing outward with
respect to $\gr_i$).
For instance, regarding $\eta_1$ we have that for every point
$(x,u(x))\in \de \gr_1 \cap \{x_{n+1}=0\}$
\begin{gather*}
\eta_1(x,u(x))\cdot e_{n+1}<0 ,
\quad
\eta_1(x,u(x)) \cdot \big(-\nabla u(x), 1\big)= 0\\
\text{and}\quad
\eta_1(x,u(x)) \cdot \big(e_i+\de_i u(x)\, e_{n+2}\big) = 0
\quad\forall\;i=1,\ldots, n.
\end{gather*}
Therefore, by taking into account that $\de_iu\cdot\de_{n+1}u=0$
on $B_1'$ for $i\in\{1,\ldots,n\}$, in view of \eqref{e.signorini bc} 
and Proposition~\ref{p.C1}, simple algebra yields that for every 
$x=(x',0)$ we have
\[
\eta_1(x,u(x)) = 
\left(0, -\frac{1}{\sqrt{1+|\de_{n+1} u(x)|^2}},
-\frac{\de_{n+1} u(x)}{\sqrt{1+|\de_{n+1} u(x)|^2}}\right).
\]
Similarly, we have 
\[
\eta_2(x,u(x)) = 
\left(0, 
\frac{1}{\sqrt{1+|\de_{n+1} u(x)|^2}},
\frac{-\de_{n+1} u(x)}{\sqrt{1+|\de_{n+1} u(x)|^2}}\right),
\]
\[
\eta_3(x,-u(x)) = 
\left(0,
-\frac{1}{\sqrt{1+|\de_{n+1}u(x)|^2}},
\frac{\de_{n+1} u(x)}{\sqrt{1+|\de_{n+1} u(x)|^2}}\right),
\]
and
\[
\eta_4(x,-u(x)) = 
\left(0, 
\frac{1}{\sqrt{1+|\de_{n+1} u(x)|^2}},
\frac{\de_{n+1} u(x)}{\sqrt{1+|\de_{n+1} u(x)|^2}}\right),
\]
where $\de_{n+1} u(x',0) = \lim_{t\downarrow 0}\frac{u(x',t)}{t}$
for every $(x',0)\in \Lambda(u)$.
Hence, using Stokes' theorem we infer that
\begin{align*}
\sum_{i=1}^4\int_{\gr_{i}} 
\div_{\gr_{i}} Y \,\d \cH^{n+1}&
= \sum_{i=1}^4\int_{\de \gr_i} Y \cdot \eta_i \, \d\cH^{n}\\
& = \sum_{i=1}^4\int_{\de \gr_i\setminus (\Lambda(u) \times \R)} Y
\cdot \eta_i \, \d\cH^{n}
+
\sum_{i=1}^4\int_{\de \gr_i \cap (\Lambda(u) \times \R)} Y \cdot 
\eta_i \, \d\cH^{n}.
\end{align*}
We couple the different terms as follows:
\begin{align*}
&\int_{\de \gr_1\setminus (\Lambda(u) \times \R)} Y \cdot 
\eta_1 \, \d\cH^{n}
+ \int_{\de \gr_2\setminus (\Lambda(u) \times \R)} Y \cdot 
\eta_2 \, \d\cH^{n}\\
&= 
\int_{B_1'\setminus \Lambda(u)} \Big[\Big(Y(x',0, u(x',0))\cdot 
\eta_1\big(x',0,u(x',0)\big) + Y(x',0, u\big(x',0)\big)\cdot 
\eta_2\big(x',0,u(x',0)\big) \Big) \cdot\\
&\qquad \qquad\qquad\qquad\cdot\sqrt{1+|\nabla' u 
(x',0)|^2}\,\Big]\,\d x'= 0,
\end{align*}
because $\eta_1(x',0,u(x',0))=-\eta_2(x',0,u(x',0)) = -e_{n+1}$
for every  $(x',0)\in B_1'\setminus \Lambda(u)$.
For the same reasons
\begin{align*}
\int_{\de \gr_3\setminus (\Lambda(u) 
\times \R)} Y \cdot \eta_3 \, \d\cH^{n}
+ \int_{\de \gr_4\setminus (\Lambda(u) \times \R)} Y \cdot 
\eta_4 \, \d\cH^{n}= 0.
\end{align*}
Next we pair
\begin{align*}
&\int_{\de \gr_1\cap (\Lambda(u) \times \R)} Y \cdot 
\eta_1 \, \d\cH^{n}
+ \int_{\de \gr_4\cap (\Lambda(u) \times \R)} Y \cdot 
\eta_4 \, \d\cH^{n}\\
&= 
\int_{\Lambda(u)} \Big(Y\big(x',0, 0\big)\cdot 
\eta_1\big(x',0,0\big) + \\
&\qquad \qquad\qquad\qquad Y\big(x',0, 0\big)\cdot 
\eta_4\big(x',0,0\big) \Big) 
\,\sqrt{1+|\nabla' u 
(x',0)|^2}\,\d x'=0,
\end{align*}
where we used that
$\eta_1\big(x',0,0\big) + 
\eta_4\big(x',0,0\big)=0$ for all $(x',0)\in B_1'$.
With a similar argument, we also have 
\begin{align*}
\int_{\de \gr_2\cap (\Lambda(u) \times \R)} Y 
\cdot 
\eta_2 \, \d\cH^{n}
+ \int_{\de \gr_3\cap (\Lambda(u) \times \R)} Y \cdot 
\eta_3 \, \d\cH^{n} = 0.
\end{align*}
Collecting the estimates above we conclude the proposition.
\end{proof}

Finally, Proposition~\ref{p.C1alfa}, Proposition~\ref{p:ob vs 2-val} 
and \cite[Theorem 7.1]{SiWi16} imply the optimal regularity 
for the solution to the thin obstacle problem.

\begin{theorem}\label{t.opt reg}
Let $g\in C^2(\R^{n+1})$ be even symmetric with respect to $x_{n+1}$
with $g\vert_{\{x_{n+1}=0\}}\geq 0$, and let 
$u\in W^{1, \infty}(B_1)$ be the solution to the thin obstacle problem
\eqref{e:ob prob}.
Then, $u \in C_{\textup{loc}}^{1, \sfrac12}(B_1^+\cup B_1')$.
\end{theorem}

\begin{proof}
By Proposition \ref{p.C1alfa} we have that 
there exists $\alpha\in (0,1)$ such that
the two valued function 
$U=\{-u,u\} \in C_{\textup{loc}}^{1,\alpha}(B_1)$,
because
$[DU]_{C^{0,\alpha}(B_1)} \leq [Du]_{C^{0,\alpha}(B_1^+)}$ 
(here $[\cdot]_{C^{0,\alpha}(E)}$ denotes the H\"older seminorm 
of the relevant function on the set $E$).
By Proposition \ref{p:ob vs 2-val} we have that the 
graph of $U$ induces a two-valued minimal graph; we are in the
position to apply \cite[Theorem 7.1]{SiWi16} and conclude that
$U \in C_{\textup{loc}}^{1,\sfrac12}(B_1)$, or equivalently 
$u \in C_{\textup{loc}}^{1, \sfrac12}(B_1^+\cup B_1')$.
\end{proof}

%
%
\section{The structure of the free boundary}\label{s:freebdry}
In this section we provide a detailed analysis of the free boundary
points for the thin obstacle problem for the area functional.
As mentioned in the introduction we prove more refined 
conclusions than those contained in Theorem~\ref{t.main},
recovering the analogous results shown for the Dirichlet energy
in \cite{AtCaSa08, FoSp17}.

To state the result we need to introduce three classes of functions $\Phi_m$,
$\Psi_m$ and $\Pi_m$ for $m\in \N\setminus\{0\}$, that are 
explicitly defined as follows:
\begin{gather}
\Phi_m(x_1,x_2):=\mathrm{Re}\Big[(x_1+i|x_2|)^{2m}\Big],
\label{e:polyn b}\\
\Psi_m(x_1,x_2):=\mathrm{Re}\Big[(x_1+i|x_2|)^{2m-\sfrac12}\Big],
\label{e:hy-geo1b}
\\
\Pi_m(x_1,x_2):=\mathrm{Im}\Big[(x_1+i|x_2|)^{2m+1}\Big].
\label{e:PI}
\end{gather}
Such families of functions exhaust the homogeneous solutions to
the thin obstacle problem with null obstacle for the Dirichlet energy
having top dimensional subspaces of invariances (cp. \cite[Appendix~A]{FoSp17}).

Moreover, we recall that $I_u(x_0,\cdot)$, $x_0\in B_1'$, denotes the frequency function defined in 
\eqref{e:freq intro} that shall be studied in the next subsection. In particular,
we shall prove that there exists finite its limit value in $0^+$ denoted in
what follows by $I_u(x_0,0^+)$ for all $x_0\in\Gamma(u)$.

The following is the main theorem.

\begin{theorem}\label{t:main precise}
Let $u$ be a solution to the thin obstacle problem 
\eqref{e:ob prob}. Then, 
\begin{itemize}
\item[(i)] $\Gamma(u)$ has locally finite $(n-1)$-dimensional
Minkowski content, \ie~for every $K \subset \subset B_1'$ there 
exists a constant $C(K)>0$ such that
\begin{equation}\label{e:misura1}
\cL^{n+1}\big(\cT_r(\Gamma(u)) \cap K \big) \leq C(K)\,r^2
\quad \forall\; r \in (0,1),
\end{equation}
where $\cT_r(E) := \{ x \in \R^{n+1}: \dist(x,E) <r\}$;

\item[(ii)] $\Gamma(u)$
is $\cH^{n-1}$-rectifiable, \ie~there exist at most countably many 
$C^1$-regular submanifolds $M_i \subset \R^n$ of dimension 
$n-1$ such that 
\begin{equation}\label{e:rect1}
\cH^{n-1}\big(\Gamma(u)) 
\setminus \cup_{i\in\N} M_i\big) = 0\,;
\end{equation}

\item[(iii)] $\Gamma_{\sfrac32}(u):=
\{x_0\in \Gamma(u):\,I_u(x_0,0^+)=\sfrac32\}$ 
is locally a $C^{1,\alpha}$ regular submanifold 
of dimension $n-1$ for some dimensional constant $\alpha>0$.
\end{itemize}
Moreover, there exists a subset $\Sigma(u)\subset \Gamma(u)$ 
with Hausdorff dimension at most $n-2$ such that
\[
I_u(x_0,0^+)\in \{2m, 2m-\sfrac 12, 2m+1\}_{m\in\N\setminus\{0\}}
\quad \forall\; x_0\in 
\Gamma(u)\setminus \Sigma(u).
\]
\end{theorem}

Theorem \ref{t:main precise} generalizes to the nonlinear
setting of minimal surfaces the known results for the regularity
of the free boundary shown for the fractional obstacle problem.
The conclusion in (iii) extends the analysis of the regular part 
of the free boundary
done in \cite{AtCaSa08} and its proof 
follows from \cite{GaSm14} as a consequence of the 
epiperimetric inequality established in \cite{FoSp16, GaSm14}.
While for the rest, 
the statements are modelled on our results in \cite{FoSp17} and 
the proof
is accomplished by the same arguments exploited for 
the Dirichlet energy in 
\cite{FoSp17}; for the sake of completeness, 
in the following we provide
the readers with the details of the needed changes.


\subsection{Obstacle problems for Lipschitz quadratic energies}\label{s:quadratic}
Given a solution 
$u$ to \eqref{e:ob prob}, it follows from
\eqref{e.signorini bc} 
that $u$ minimizes 
the following thin obstacle problem for a 
specific quadratic energy:
\begin{equation}\label{e:Aenrg}
\mathcal{Q}:\mathcal{A}_g\ni v\longmapsto\frac12\int_{B_1}\vartheta(x)
|\nabla v(x)|^2\, dx, 
\qquad \text{with }\;
\vartheta(x):=\big(1+|\nabla u(x)|^2\big)^{-\sfrac12}.
\end{equation}
Note that the above functional is coercive because
\begin{equation}\label{e:Acoercive}
0<\big(1+L^2\big)^{-\sfrac12}\leq \vartheta(x)\leq  1\,,
\end{equation}
where as usual $L=\Lip(u)$.
Actually, $\vartheta(x)=1$ if $x\in \Gamma(u)$.
Moreover, we have that $\vartheta$ is a Lipschitz function, 
as proven in the following lemma.

\begin{lemma}\label{l:thetalip}
Let $u$ be a solution to the thin obstacle problem 
\eqref{e:ob prob}, then $\vartheta \in W^{1,\infty}(B_1)$.
\end{lemma}
\begin{proof}
Setting $d(x) := \dist(x, \Gamma(u))$, $x\in B_1$,
by the regularity result in Theorem~\ref{t.opt reg}
and the classical Schauder estimates we deduce that 
\[
|u(x)| \leq C\,d^{\sfrac32}(x), \quad
|\nabla u(x)| \leq C\,d^{\sfrac12}(x) \quad
\text{and}\quad |\nabla^2 u(x)| \leq C\,d^{-\sfrac12}(x),
\]
for some constant $C>0$, and therefore
\[
|\nabla \vartheta| = \big(1+|\nabla u|^2\big)^{-\sfrac32}
\, |\nabla^2u\, \nabla u| \leq C.\qedhere
\]
\end{proof}

The basic idea of exploiting the regularity of $u$ itself to reduce the problem to quadratic 
energies with Lipschitz coefficients has been recently considered in the literature 
for the classical obstacle problem (see, {\eg}, \cite{FoGeSp15, MoSp17, FoGerSp18}).

\subsection{Frequency function}\label{s:frequency}
Given the Lipschitz continuity of $\vartheta$ we can 
prove monotonicity of the following frequency type function at a point $x_0 \in B_1'$ 
defined by
\[
I_u(x_0,t) := \frac{r D_u(x_0,t)}{H_u(x_0,t)}
\quad  \forall\; r<1-|x_0|,
\]
where
\[
D_u(x_0,t) := \int \phi\big(\textstyle{\frac{|x-x_0|}{t}}\big)\vartheta(x)\,|\nabla u(x)|^2\d x,
\]
and
\[
H_u(x_0,t) := - \int \phi'\Big(\textstyle{\frac{|x-x_0|}{t}}\Big)
\,\vartheta(x)\frac{u^2(x)}{|x-x_0|}\,\d x.
\]
(see \cite{DS-memo} for the first use of this variation of Almgren's frequency function).
Here, $\phi:[0,+\infty) \to [0,+\infty)$ is the function given by
\[
\phi(t) := 
\begin{cases}
1 & \text{for } \; 0\leq t \leq \frac{1}{2},\\
2\,(1-t) & \text{for } \; \frac{1}{2} < t \leq 1,\\
0 & \text{for } \; 1 < t .\\
\end{cases}
\]
It is also useful to introduce
\[
E_u(x_0,t):=
-\int \phi'\Big(\textstyle{\frac{|x-x_0|}{t}}\Big)
\vartheta(x)\frac{|x-x_0|}{t^2}
\Big(\nabla u(x) \cdot\frac{x-x_0}{|x-x_0|}\Big)^2\, \d x.
\]
In what follows, we shall not highlight the dependence on the 
base point $x_0$ in the quantities above if it coincides with 
the origin.

By exploiting the integration by parts formulas
used in \cite{FoGeSp15}, we show the following
variant of the monotonicity formula for the frequency.

\begin{proposition}\label{p:monotonia+lower}
Let $u$ be a solution to the thin obstacle problem 
\eqref{e:ob prob} in $B_1$. 
Then, there exists a nonnegative constant
$C_{\ref{p:monotonia+lower}}$ 
depending on $\Lip(u)$, such that for all $x_0 \in B_1'$, 
and for $\mathcal{L}^1$ a.e. $t\in(0,1-|x_0|)$
\begin{equation}\label{e:Iprime}
I_u'(x_0,t)=\frac{2t}{H_u^2(x_0,t)}
\big(H_u(x_0,t)E_u(x_0,t)-D_u^2(x_0,t)\big)+R_u(x_0,t),
\end{equation}
with $|R_u(x_0,t)|\leq C_{\ref{p:monotonia+lower}} I_u(x_0,t)$.
In particular, the function $(0,1-|x_0|)\ni t\mapsto
e^{C_{\ref{p:monotonia+lower}}t}I_u(x_0,t)$ is nondecreasing and
\begin{equation}\label{e:monotonia freq}
e^{C_{\ref{p:monotonia+lower}}r_1}I_u(x_0,r_1) -
e^{C_{\ref{p:monotonia+lower}}r_0}I_u(x_0,r_0)  \geq
\int_{r_0}^{r_1}\frac{2\,t\,
e^{C_{\ref{p:monotonia+lower}}t}}{H_u^2(x_0,t)}
\Big(H_u(x_0,t) \; E_u(x_0,t) - D_u^2(x_0,t) \Big)\,\d t,
\end{equation} 
for $0 < r_0 < r_1 < 1 - |x_0|$,
and the limit 
$I_u(x_0,0^+)=\lim_{t\downarrow 0}I_u(x_0,t)$ exists finite.
\end{proposition}

\begin{proof}
We need to estimate the derivatives of $D_u$ and $H_u$:
by exploiting the integration by parts formulas
used in \cite{FoGeSp15} one can show that
for every $x_0\in B_1'$ and for $\cL^1$ a.e.
$r\in(0,1-|x_0|)$,
\begin{align}\label{e:Dprime}
D_u^\prime(x_0,r) = {} &
\frac{n-1}r D_u(x_0,r)+ 2E_u(x_0,r)+\eps_D(x_0,r),
\end{align} 
and
\begin{equation}\label{e:Hprime}
H_u^\prime(x_0,r)=\frac nr H_u(x_0,r)
+2D_u(x_0,r)+\eps_H(x_0,r),
\end{equation} 
with $|\eps_D(x_0,r)| \leq C\,D_u(x_0,r)$ and
$|\eps_H(x_0,r)|\leq CH_u(x_0,r)$
for some constant $C>0$ depending on $\Lip(u)$.
Moreover, for all  $0<r<1-|x_0|$,
\begin{equation}\label{e:D}
D_u(x_0,r)=-\textstyle{\frac1r}
\int \phi'\big(\textstyle{\frac{|x-x_0|}{r}}\big)\,\vartheta(x)u(x)
\nabla u(x)\cdot\frac{x-x_0}{|x-x_0|}\,\d x.
\end{equation}
The details of \eqref{e:Dprime}, \eqref{e:Hprime} and
\eqref{e:D} are postponed to the appendix.

For the sake of simplicity assume $x_0=\underline{0}$ (recall that 
in this case we drop the dependence on the base point in the relevant quantities).
By \eqref{e:Hprime} and \eqref{e:Dprime}, 
we compute the derivative of $\log I_u(t)$ as follows: 
\begin{equation*}
\frac{I_u'(t)}{I_u(t)}=\frac 1t+\frac{D_u'(t)}{D_u(t)}-
\frac{H_u'(t)}{H_u(t)}
 =2\,\frac{E_u(t)}{D_u(t)}-2\frac{D_u(t)}{H_u(t)}
 +\frac{\eps_D(t)}{D_u(t)}-\frac{\eps_H(t)}{H_u(t)}.
\end{equation*}
Hence, being $|\eps_D(t)|\leq C D_u(t)$ and 
$|\eps_H(t)|\leq CH_u(t)$, \eqref{e:Iprime} readily follows.
In addition, 
\[
I_u'(t)+C_{\ref{p:monotonia+lower}} I_u(t)\geq\,\frac{2t}{H_u^2(t)}
\big(H_u(t)E_u(t)-D_u^2(t)\big),
\]
thus leading to inequality \eqref{e:monotonia freq}
by multiplying with $e^{C_{\ref{p:monotonia+lower}}t}$ and integrating.
Finally, by \eqref{e:D} and the Cauchy--Schwarz inequality, 
the map $t\mapsto e^{C_{\ref{p:monotonia+lower}}t}I_u(t)$ is non-decreasing. 
\end{proof}

We also derive additive quasi-monotonicity formula for the frequency.
\begin{corollary}\label{c:monotonia add}
Let $u$ be a solution to the thin obstacle problem 
\eqref{e:ob prob} in $B_1$. 
For every $A>0$ there exists $C_{\ref{c:monotonia add}} =
C_{\ref{c:monotonia add}}(\Lip(u),A)>0$ such that 
for all $x_0 \in B_1'$ with $I_u(x_0,r)\leq A$,
$r\in(0,1-|x_0|)$, then
\begin{equation}\label{e:monotonia add}
(0,r]\ni t\longmapsto I_u(x_0,t)+C_{\ref{c:monotonia add}}t
\quad\textrm{is nondecreasing.}
\end{equation}
\end{corollary}
\begin{proof}
Proposition \ref{p:monotonia+lower} implies that 
$I_u(x_0,t)\leq e^{C_{\ref{p:monotonia+lower}}}A$ for 
all $t\in(0,r]$.
Therefore, from inequality \eqref{e:Iprime} and the estimate on the rest $R_u(x_0,t)$, 
we deduce the conclusion with $C_{\ref{c:monotonia add}}:=C_{\ref{p:monotonia+lower}}\,
e^{C_{\ref{p:monotonia+lower}}\,A}$. 
\end{proof}

\subsection{Lower bound on the frequency and compactness}
The frequency of a solution to \eqref{e:ob prob} 
at free boundary points is bounded from below by a universal constant.
A preliminary lemma is needed.
\begin{lemma}\label{l:Poinc}
Let $u$ be a solution to the thin obstacle problem 
\eqref{e:ob prob} in $B_1$. Then, there exists a constant 
$C=C(n,[\nabla u]_{C^{0,\sfrac12}(B_{\sfrac34}^+)})>0$ such that for every 
$x_0\in \Gamma(u)\cap B_{\sfrac14}$ and for every $0<r<\sfrac12$
\begin{equation}\label{e:Poinc2}
\int_{\partial B_r(x_0)}|u(x)|^2\d x\leq 
C r \int_{B_r(x_0)}|\nabla u(x)|^2\d x+C r^{n+3}.
\end{equation}
\end{lemma}
\begin{proof}
By Poincar\'e-Wirtinger inequality we have 
\begin{equation}\label{e:pw}
 \int_{\partial B_r(x_0)}|u(x)|^2\d \mathcal{H}^n\leq C r \int_{B_r(x_0)}|\nabla u(x)|^2\d x
 +Cr^n\Big(\fint_{\partial B_r(x_0)}u(x)\,\d \mathcal{H}^n\Big)^2,
\end{equation}
for some dimensional constant $C>0$.
To estimate the mean value of $u$ we argue as follows.
By direct calculation
 \[
  \frac{\d}{\d r}\Big(\fint_{B_r(x_0)}u(x)\,\d x\Big)=
  \fint_{B_r(x_0)}\langle\nabla u(x),\frac{x}{r}\rangle\d x.
 \]
Therefore, recalling that $\nabla u(x_0)=\underline{0}$ since $x_0\in\Gamma(u)$, 
by one-sided $C^{1,\sfrac12}$ regularity we find that
\begin{align}\label{e:primo termine}
\left| \frac{\d}{\d r}\Big(\fint_{B_r(x_0)}u(x)\,\d x\Big)\right|
\leq\fint_{B_r(x_0)}|\nabla u(x)|\d x\leq C r^{\sfrac12}
\end{align}
with $C=C(n,[\nabla u]_{C^{0,\sfrac12}(B_{\sfrac34}^+)})>0$. 
Hence, recalling that $u(x_0)=0$, by integration we infer that 
\begin{align}\label{e:secondo termine}
 \left|\fint_{B_r(x_0)}u(x)\,\d x\right|\leq C\,{r^{\sfrac32}}.
\end{align}
Finally, noting that
\[
  \frac{\d}{\d r}\Big(\fint_{B_r(x_0)}u(x)\,\d x\Big)=
 \frac{n+1}{r}\Big(\fint_{\partial B_r(x_0)}u(x)\,\d \mathcal{H}^n
 -\fint_{B_r(x_0)}u(x)\,\d x\Big)
 \]
we conclude from \eqref{e:primo termine} and \eqref{e:secondo termine} that
for some $C=C(n,[\nabla u]_{C^{0,\sfrac12}(B_{\sfrac34}^+)})>0$ we have
 \[
 \left|\fint_{\partial B_r(x_0)}u(x)\,\d \mathcal{H}^n\right|
 \leq Cr^{\sfrac32}.
 \]
 In turn, the latter inequality and \eqref{e:pw} yield \eqref{e:Poinc2}.
\end{proof}

A first rough bound from below on the frequency then easily follows.
\begin{lemma}\label{l:freq lower}
Let $u$ be a solution to the thin obstacle problem 
\eqref{e:ob prob} in $B_1$.  
There exist a constant $C_{\ref{l:freq lower}}=C_{\ref{l:freq lower}}
(n,L, [\nabla u]_{C^{0,\sfrac12}(B_{\sfrac34}^+)})>0$ and a radius
$r_{\ref{l:freq lower}}=r_{\ref{l:freq lower}}(n,L,[\nabla u]_{C^{0,\sfrac12}(B_{\sfrac34}^+)})\in(0,\sfrac12)$
such that, for every $x_0 \in \Gamma(u)\cap B_{\sfrac14}$ we have for all $r\in (0, r_{\ref{l:freq lower}})$
\begin{equation}\label{e:freq lower}
I_u(x_0,r) \geq C_{\ref{l:freq lower}}.
\end{equation}
\end{lemma}
\begin{proof}
The co-area formula and an integration by parts give 
\begin{align}\label{e:rappr H}
H_u(x_0,r) = 2\int_{\frac{r}{2}}^r \frac{dt}{t}
\int_{\de B_t(x_0)} \vartheta(x)|u(x)|^2\d \cH^{n}(x),
\end{align}
and
\begin{align}\label{e:rappr D}
D_u(x_0,r) = \frac{2}{r} \int_{\frac{r}{2}}^r dt\int_{B_t(x_0)} 
\vartheta(x)|\nabla u(x)|^2\,\d x
\end{align}
(cf. \cite[Lemma 2.9]{FoSp17}). Moreover, recalling that $\vartheta\in\Lip(B_1)$ with 
$\vartheta(x_0)=1$ as $x_0\in\Gamma(u)$ and $0<\vartheta(x)\leq 1$ for all $x\in B_1'$, 
we conclude from \eqref{e:rappr H} that 
\begin{align*}
&\left|H_u(x_0,r) - 2\int_{\frac{r}{2}}^r \frac{dt}{t}
\int_{\de B_t(x_0)} |u(x)|^2\d \cH^{n}(x)\right|
\leq 2L\int_{\frac{r}{2}}^r dt\int_{\de B_t(x_0)} |u(x)|^2\d \cH^{n}(x)\notag\\
&=2L\int_{B_r(x_0)\setminus B_{\sfrac r2}(x_0)}|u(x)|^2\,\d x 
\stackrel{\eqref{e:L2 vs H}}{\leq}
2L\sqrt{1+L^2}\,e^{C_{\ref{p:hprime}}r}r\,H_u(x_0,r).
\end{align*}
Hence, we find that 
\begin{equation}\label{e:stima H}
 H_u(x_0,r) \leq 4\int_{\frac{r}{2}}^r \frac{dt}{t}
\int_{\de B_t(x_0)} |u(x)|^2\d \cH^{n}(x)
\end{equation}
provided that $0<r\leq r_{\ref{l:freq lower}}\leq
(4L\sqrt{1+L^2}\,e^{C_{\ref{p:hprime}}})^{-1}$.

Analogously, by taking into account \eqref{e:rappr D} we have that
\begin{align*}
&\left|D_u(x_0,r) - \frac2r\int_{\frac{r}{2}}^r dt
\int_{B_t(x_0)} |\nabla u(x)|^2\d \cH^{n}(x)\right|
\leq 2L\int_{\frac{r}{2}}^r dt\int_{B_t(x_0)} |\nabla u(x)|^2\d \cH^{n}(x),
\end{align*}
from which we deduce that 
\begin{equation}\label{e:stima D}
\frac2r\int_{\frac{r}{2}}^r dt\int_{B_t(x_0)} |\nabla u(x)|^2\d \cH^{n}(x)
\leq 4D_u(x_0,r)
\end{equation}
as $0<r\leq r_{\ref{l:freq lower}}\leq\frac 1{2L}$.

In particular, from the Poincar\'e inequality \eqref{e:Poinc2} and estimates 
\eqref{e:stima H}, \eqref{e:stima D} we get for some constant $C=C(n,[\nabla u]_{C^{0,\sfrac12}(B_{\sfrac34}^+)})>0$
\[
H_u(x_0,r) \leq C\,r D_u(x_0,r) + C\, r^{n+3}\,,
\]
for all $r\in(0,r)$ provided that $r_{\ref{l:freq lower}}\leq 
(4L\sqrt{1+L^2}\,e^{C_{\ref{p:hprime}}})^{-1}\wedge\sfrac12$.
Then, either $e^{C_{\ref{p:monotonia+lower}}\,r}I_u(x_0,r)\geq 1$ for every 
$r\in (0, r_{\ref{l:freq lower}})$, from which we infer 
$I_u(x_0,r)\geq e^{-C_{\ref{p:monotonia+lower}}}$ for all $r\in (0, r_{\ref{l:freq lower}})$; 
or $e^{C_{\ref{p:monotonia+lower}}\,r}I_u(x_0,r)<1$ for all $r\in (0,r_{x_0})$, 
$r_{x_0}< r_{\ref{l:freq lower}}$, by Proposition~\ref{p:monotonia+lower}.
In the last instance, $I_u(x_0,r)\geq e^{-C_{\ref{p:monotonia+lower}}}$ 
for all $r\in [r_{x_0}, r_{\ref{l:freq lower}})$, and $I_u(x_0,r)<1$ for every $r\in (0,r_{x_0})$. 
Thus, 
we have that $H_u(x_0,r)\geq e^{-C_{\ref{p:hprime}}} H_u(x_0,r_{x_0})\,\Big(\frac r{r_{x_0}}\Big)^{n+1}$ 
for all radii $r\in(0,r_{x_0})$ (cf. \eqref{e:H1} in the appendix). In particular, for such radii 
we conclude that
\[
I_u(x_0,r)\geq \frac1C-e^{C_{\ref{p:hprime}}}\frac{r_{x_0}^{n+1}}{H_u(x_0,r_{x_0})}\,r^2\,,
\]
and thus there exists $\rho_{x_0}\leq r_{x_0}$ such that
\[
I_u(x_0,r)\geq \frac1{2C}
\]
for all $r\in(0,\rho_{x_0})$. In turn, this and the quasi-monotonicity of the frequency 
in Proposition~\ref{p:monotonia+lower} yield that for all $r\in(0, r_{\ref{l:freq lower}})$
\[
I_u(x_0,r)\geq \frac{e^{-C_{\ref{p:monotonia+lower}}}}{2C}\,.\qedhere
\]
\end{proof}

\subsection{Blowup profiles}
An important consequence of the quasi-monotonicity of the frequency
in Proposition~\ref{p:monotonia+lower} and of the universal lower bound
for the frequency in Lemma~\ref{l:freq lower} is the existence of 
nontrivial blowup profiles.
For $u:B_1\to\R$ solution of \eqref{e:ob prob} we introduce the rescalings
\begin{equation}\label{e:rescaling0}
u_{x_0,r} (y) := 
\frac{r^{\sfrac{n}{2}}\,u(x_0+ry)}{H^{\sfrac12}(x_0,r)}
\quad \forall \; r \in(0,1 - |x_0|), \;\forall\; y \in 
B_{\frac{1-|x_0|}{r}}.
\end{equation}
By the same arguments exploited in the blowup analysis in \cite[Section~2.5]{FoSp17},
for every $x_0 \in \Gamma(u)$ and for every sequence
of numbers $(r_j)_{j\in\N} \subset (0,1-|x_0|)$ with
$r_j \downarrow 0$, there exist a subsequence
$(r_{j_k})_{k\in \N}$ and a function $u_0 \in W^{1,2}_{\loc}(\R^{n+1})$ such that
\begin{equation}\label{e:blow-up}
u_{x_0,r_{j_k}} \to u_0 \quad \text{in }\;W^{1,2}_{\loc}(\R^{n+1}).
\end{equation}
Moreover, $u_0$ is the solution to the Signorini problem for the Dirichlet energy  
on $\R^{n+1}$, \ie ~satisfying
\begin{gather}\label{e:Dirichlet bc}
\begin{cases}
\triangle u_0 = 0 & \text{in }\; \{x_{n+1}>0\},\\
\de_{n+1} u_0 \leq 0 \quad\text{and}\quad
u_0\,\de_{n+1} u_0 = 0 & \text{on }\; \{x_{n+1}=0\},
\end{cases}
\end{gather}
and $u_0$ is $I_u(x_0,0^+)$-homogeneous, 
because by rescaling $I_{u_0}(\underline{0},r) = I_u(x_0,0^+)$ for every $r>0$.

In particular, the classification of the blowup profiles 
is the same as for the Dirichlet energy, and consists in the functions
$\Phi_m$, $\Psi_m$, $\Pi_m$ in \eqref{e:polyn b}, \eqref{e:hy-geo1b}
and \eqref{e:PI} in case the subspace of invariant directions has maximal 
dimension.

\subsection{Spatial oscillation for the frequency}
Next we recall the basic estimate on the spatial 
oscillation of the frequency which is at the heart of the analysis
in \cite{FoSp17}.
We introduce the notation: for a point 
$x \in B_1'$ and a radius $0<\rho<r$, we set
\[
\Delta^r_{\rho}(x) := I_u(x,r)+C_{\ref{c:monotonia add}}r
- I_u(x,\rho)-C_{\ref{c:monotonia add}}\rho.
\]
Note that $\Delta^r_{\rho}(x)\geq 0$ in view of Corollary~\ref{c:monotonia add}.

The following proposition is a straightforward extension of
\cite[Proposition~3.3]{FoSp17}.

\begin{proposition}\label{p:D_x frequency}
For every $A>0$ there exists 
$C_{\ref{p:D_x frequency}}(n,\Lip(u),A)>0$ such that,
if $\rho>0$, $R>9$
and $u : B_{4R\rho}(x_0) \to \R$ is a solution to the thin
obstacle problem \eqref{e:ob prob} 
in $B_{4R\rho}(x_0)$, with $x_0\in\Gamma(u)$ and 
$I_u(x_0,4R\rho) \leq A$, then
\begin{equation}\label{e:D_x frequency}
\big\vert I_u\big(x_1, R\rho\big) - I_u\big(x_2, R\rho\big)\big\vert
\leq C_{\ref{p:D_x frequency}} \,
\left[\Big(\Delta^{2(R+2)\rho}_{\sfrac{(R-4) \rho}{2}}(x_1)\Big)^{\sfrac{1}{2}} 
+ 
\Big(\Delta^{2(R+2)\rho}_{\sfrac{(R-4)\rho}{2}}(x_2)\Big)^{\sfrac{1}{2}}\right]+C_{\ref{p:D_x frequency}}R\rho ,
\end{equation}
for every $x_1, x_2 \in B'_{\rho}$.
\end{proposition}

\begin{proof}
The proof is a variant of \cite[Proposition~3.3]{FoSp17}.
For readers' convenience, we repeat some of the arguments with
the necessary changes.

Without loss of generality, we consider $x_0= \underline{0}$.
With fixed $x_1, x_2 \in B'_{\rho}$, let 
$x_t:=tx_1+(1-t)x_2$, $t\in[0,1]$, and 
consider the map $t\mapsto I_u(x_t,R \rho)$.
Set $e:=x_1-x_2$, then $e \cdot e_{n+1} = 0$.
Since the functions $x\mapsto H_u(x,R \rho)$ 
and $x\mapsto D_u(x,R \rho)$ are differentiable, 
we get
\begin{equation}\label{e:Iint}
I_u(x_1,R \rho) - I_u(x_2,R \rho)
=\int_0^1\partial_e I_u(x_t,R\rho)\,\d t.
\end{equation}
To compute the last integrand, we start off with noting that 
for all $\lambda\in\R$
\begin{align}\label{e:dH}
\de_e H_u(x_t,R\rho) = &-  \, \int \phi'\big(\textstyle{\frac{|y|}{R\rho}}\big)\,\frac{1}{|y|} 
\big(2\vartheta(y+x_t)\,u(y+x_t)\, \de_e u(y+x_t)+\de_e \vartheta(y+x_t)\,u^2(y+x_t)\big)\,\d y\notag\\
=&- 2\int \phi'\big(\textstyle{\frac{|y|}{R\rho}}\big)\,\frac{1}{|y|} 
\vartheta(y+x_t)\,u(y+x_t)\big(\de_e u(y+x_t)-\lambda u(y+x_t)\big)\,\d y\notag\\
&+2\lambda H_u(x_t,R\rho)-\int \phi'\big(\textstyle{\frac{|y|}{R\rho}}\big)\,\de_e \vartheta(y+x_t)\,\frac{u^2(y+x_t)}{|y|}
\,\d y,
\end{align}
and by Proposition~\ref{p.H2}
\begin{align}\label{e:dD}
 \de_e D_u(x_t,R\rho) & = 
 \int \phi\big(\textstyle{\frac{|y|}{R\rho}}\big)\big(2\vartheta(y+x_t)
\nabla u(y+x_t)  \cdot \nabla (\de_e u)(y+x)\, +\de_e\vartheta(y+x_t)|\nabla u(y+x_t) |^2\big)
\,\d y\notag\\
& = - \frac{2}{R\rho}\int \phi'\big(\textstyle{\frac{|y|}{R\rho}}\big)\,
\vartheta(y+x_t)\de_e u(y+x_t)\, \nabla u(y+x_t) \cdot\frac{y}{|y|}\,\d y\notag\\
& + \int \phi\big(\textstyle{\frac{|y|}{R\rho}}\big)\de_e\vartheta(y+x_t)|\nabla u(y+x_t) |^2
\,\d y\notag\\
& \stackrel{\eqref{e:D}}{=} - \frac{2}{R\rho}
 \int \phi'\big(\textstyle{\frac{|y|}{R\rho}}\big)\,
\vartheta(y+x_t)\big(\de_e u(y+x_t)-\lambda u(y+x_t)\big)
 \nabla u(y+x_t) \cdot\frac{y}{|y|}\,\d y\notag\\
& +2\lambda D_u(x_t,R\rho)
+\int \phi\big(\textstyle{\frac{|y|}{R\rho}}\big)
\de_e\vartheta(y+x_t)|\nabla u(y+x_t) |^2\,\d y.
\end{align}
To deduce the second equality we have applied the 
divergence theorem to the vector field 
$V(y):= \phi\big(\textstyle{\frac{|y|}{R\rho}}\big)\,
\vartheta(x+y)\,\de_e u(y+x)\, \nabla u(y+x)$
(note that $V \in C^\infty(B_{R\rho}\setminus B_{R\rho}')$, 
$V$ has zero trace on $\partial B_{R\rho}$ and the 
divergence of $V$ does not concentrate on $B_1'$).

Then, by formulas \eqref{e:dH} and \eqref{e:dD}, 
we have that
\begin{align}\label{e:dI}
 \partial_e & I_u(x_t,R \rho)  = I_u(x_t,R \rho)
 \left(\frac{\partial_eD_u(x_t,R \rho)}{D_u(x_t,R \rho)}-
 \frac{\partial_eH_u(x_t,R \rho)}{H_u(x_t,R \rho)}\right)\notag\\
&=\textstyle{\frac{2}{H_u(x_t,R \rho)}}
 \int \phi'\big(\textstyle{\frac{|z-x_t|}{R\rho}}\big)\,\frac{\vartheta(z)}{|z-x_t|}
\big(\de_e u(z)-\lambda u(z)\big)
\big(\nabla u(z) \cdot (z-x_t)- I_u(x_t,R \rho) u(z)\big)\,\d z\notag\\
&+\textstyle{\frac{I_u(x_t,R \rho) }{D_u(x_t,R\rho)}}
\int \phi\big(\textstyle{\frac{|z-x_t|}{R\rho}}\big)
\de_e\vartheta(z)|\nabla u(z)|^2\,\d z
-\textstyle{\frac{I_u(x_t,R \rho) }{H_u(x_t,R\rho)}}
\int \phi'\big(\textstyle{\frac{|z-x_t|}{R\rho}}\big)\,\de_e \vartheta(z)\frac{u^2(z)}{|z-x_t|}\,\d z\notag\\
&=: J_t^{(1)}+J_t^{(2)}+J_t^{(3)}.
\end{align}
The estimate of $J_t^{(1)}$ is at all analogous to the estimate
in \cite[Proposition~3.3]{FoSp17} and yields
\begin{align}\label{e:J1t}
J_t^{(1)} \leq&
C\big(\triangle_{\sfrac{R \rho}2-2\rho}^{2(R+2)\rho}(x_1)\big)^{\sfrac 12}
+C \big(\triangle_{\sfrac{R \rho}2-2\rho}^{2(R+2)\rho}(x_2)\big)^{\sfrac 12}
\end{align}
Recalling that $\vartheta$ is 
Lipschitz continuous, for $J_t^{(2)}$ and $J_t^{(3)}$
we get that there exists a constant $C=C(\Lip(u),A)>0$ such that
\begin{equation}\label{e:J2t}
|J_t^{(2)}+J_t^{(3)}|\leq C|x_1-x_2|.
\end{equation}
By collecting \eqref{e:Iint}, \eqref{e:dI}, \eqref{e:J2t} and \eqref{e:J1t} 
we conclude.
\end{proof}

\subsection{Proof of Theorem~\ref{t:main precise}}\label{s:misura}
For the proof of the main theorem is now a straightforward
adaptation of the arguments in \cite{FoSp17}.
We omit it the details and only recall the main steps of the proof.

\subsubsection{Mean-flatness}\label{s:mean-flatness}
Using the estimate on the spatial oscillation of the frequency
in Proposition \ref{p:D_x frequency}, one can easily
prove the analog of \cite[Proposition~4.2]{FoSp17}:
for every $A>0$ and $R>6$ there exists a constant 
$C>0$ such that if $u$ is a solution to
\eqref{e:ob prob} in $B_{(4R+10)r}(x_0)$, 
with $x_0 \in \Gamma(u)$ and with 
$I_u(x_0,(4R+10)r) \leq A$, then for every 
$\mu$ finite Borel measure with $\spt(\mu)\subseteq\Gamma(u)$ and
for all $p \in \Gamma(u)\cap B_{r}'(x_0)$ we have
\begin{equation}\label{e:mean-flatness vs freq}
\beta_{\mu}^2 (p,r) \leq 
\frac{C}{r^{n-1}}
\Big(\int_{B_{r}(p)} 
\Delta_{(R-5)\,\sfrac{r}{2}}^{(2R+4)\,r}(x)\,\d\mu(x)+
r^{2}\mu(B_{r}(p))\Big),
\end{equation}
where the \emph{mean flatness} of $\mu$ is defined by 
%
\begin{equation}\label{e:beta}
\beta_\mu(x,r) := \inf_{\cL} \Big(
r^{-n-1} \int_{B_r(x)} \dist^2(y,\cL)\d\mu(y)\Big)^{\sfrac{1}{2}},
\end{equation}
the infimum being taken among all affine $(n-1)$-dimensional planes $\cL \subset \R^{n+1}$.

\subsubsection{Rigidity of homogeneous solutions}\label{s:rigidity}
We set for $x_0\in B_1'$ and $t<1-|x_0|$
\[
J_u (x_0,t) := e^{C_{\ref{p:monotonia+lower}}t}I_u(x_0,t)
\]
and given $\eta, r>0$, $4r<1-|x_0|$, we say that a
solution $u:B_{4r}(x_0)\to\R$, $x_0\in\{x_{n+1}=0\}$, 
to the thin obstacle problem \eqref{e:ob prob} is
\textit{$\eta$-almost homogeneous in $B_{4r}(x_0)$} if
\[
J_u(x_0,r) - J_u\big(x_0,\sfrac{r}2\big)\leq \eta.
\]
Then, by the compactness argument in
\cite[Proposition~5.6]{FoSp17}, the following 
rigidity property holds:
for every $\tau,A>0$ there exist $\eta>0$ and 
$r_0>0$ such that, 
if $r<r_0$ and $u:B_{4r}(x_0)\to\R$, with $x_0\in\{x_{n+1}=0\}$, 
is a $\eta$-almost homogeneous solution 
in $B_{4r}(x_0)$ of the thin obstacle problem
\eqref{e:ob prob} with $x_0\in \Gamma(u)$ and 
$J_u(x_0,4r)\leq A$, then
\begin{itemize}
\item[(i)] either for every point $x\in \Gamma(u)\cap
B_{2r}(x_0)$ we have 
\begin{align}\label{e:rigidity1}
\left\vert J_{u}(x,r) - J_{u}(x_0,r)\right\vert\leq\tau,
\end{align}
\item[(ii)] or there exists a linear subspace $V\subset\R^{n}\times\{0\}$
of dimension $n-2$ such that
\begin{align}\label{e:rigidity2}
\begin{cases}
y\in \Gamma(u)\cap B_{2r}(x_0),\\
J_{u}(y,r) - J_{u}(y,\sfrac{r}2)\leq \eta
\end{cases}
\quad\Longrightarrow\quad \dist(y,V)<\tau r.
\end{align}
\end{itemize}

\subsubsection{Proof of Theorem~\ref{t:main precise}}
Finally, the main results can be obtained by following
verbatim \cite[Sections~6--8]{FoSp17} (see also \cite{FoSp17bis}). Indeed, 
\cite[Proposition~6.1]{FoSp17}, that leads to the local finiteness 
of the Minkowskii content in item (i) of Theorem~\ref{t:main precise},
is based on a covering argument that exploits 
the lower bound on the frequency in  Lemma~\ref{l:freq lower}, 
the rigidity of almost homogeneous solutions in Subsection~\ref{s:rigidity},
the control of the mean oscillation via the frequency in 
Subsection \ref{s:mean-flatness} and the discrete 
Reifenberg theorem by Naber \& Valtorta \cite[Theorem~3.4]{NaVa1}.
 
Similarly, the $\cH^{n-1}$-rectifiability of $\Gamma(u)$ in 
Theorem~\ref{t:main precise}  (ii)
is a consequence of the rectifiability criterion by
Azzam \& Tolsa \cite[Theorem~1.1]{AzTo15} and Naber \& Valtorta
\cite[Theorem~3.4]{NaVa1} together with the estimate in
Subsection~\ref{s:mean-flatness} and 
item (i) of Theorem~\ref{t:main precise} itself.

The $C^{1,\alpha}$-regularity of $\Gamma_{\sfrac 32}(u)$
follows from the approach via an epiperimetric inequality
\cite{GaSm14} being $\vartheta$ Lipschitz continuous (see also
\cite{FoSp16} for the proof of the epiperimetric inequality).

Finally, the classification of blow-up limits 
is exactly that stated in \cite[Theorem~1.3]{FoSp17}, and proved in 
\cite[Section~8]{FoSp17} (see also \cite{FoSp17bis}).

\appendix
\section{Variation formulas}
In this section we show the computations for the monotonicity 
of the frequency based on the integration formulas exploited in 
\cite{FoGeSp15} for the classical obstacle problem.


\begin{proposition}\label{p:eprime}
Let $u$ be a solution to the thin obstacle problem 
\eqref{e:ob prob} in $B_1$.
There exists a non negative constant $C_{\ref{p:eprime}}$ depending on $\Lip(u)$, 
such that for every $x_0\in B_1'$ and for $\cL^1$ a.e. $r\in(0,1-|x_0|)$,
\begin{align}\label{e:eprime-A}
D_u^\prime(x_0,r) = {} &
\frac{n-1}r D_u(x_0,r)+ 2E_u(x_0,r)+\eps_D(x_0,r),
\end{align} 
with $|\eps_D(x_0,r)| \leq C_{\ref{p:eprime}}\,D_u(x_0,r)$.

Moreover, for all  $0<r<1-|x_0|$,
\begin{equation}\label{e:D-A}
D_u(x_0,r)=-\textstyle{\frac1r}
\int \phi'\big(\textstyle{\frac{|x-x_0|}{r}}\big)\,\vartheta(x)u(x)
\nabla u(x)\cdot\frac{x-x_0}{|x-x_0|}\,\d x.
\end{equation}
\end{proposition}
\begin{proof}
Without loss of generality we may assume $x_0=\underline{0}$.
By direct differentiation we have
\begin{equation}\label{e:Dprime-A}
D_u'(r)=-\int \phi'\Big(\textstyle{\frac{|x|}{r}}\Big)
\frac{|x|}{r^2}\vartheta(x)\,|\nabla u(x)|^2\,\d x.
\end{equation}
Consider the vector field $W \in C^\infty(B_r\setminus B_r',\R^{n+1})$ defined by
\[
W(x) := \phi\Big(\textstyle{\frac{|x|}{r}}\Big)\vartheta(x)
\left(\frac{|\nabla u|^2}{2}x-(\nabla u\cdot x)\nabla u\right),
\]
and note that $W\in C^{0,\sfrac12}_{\mathrm{loc}}\cap W^{1,2}_{\mathrm{loc}}(B_1^\pm\cup B_1',\R^{n+1})$
by the regularity of $u$ (cf. Proposition~\ref{p.H2}, Theorem~\ref{t.opt reg} and Lemma~\ref{l:thetalip}).
Then, the distributional divergence of $W$ is a measure that might 
have a singular part concentrated on $B_r'$ by the trace theorem in $W^{1,2}$. On the other hand, 
recalling that $u\,\partial_{n+1}u=0$ on $B_1'$ we find
$W (x',0^\pm) \cdot e_{n+1} = 0$ for all $(x',0) \in B_r'$.
Therefore, since $u$ minimizes \eqref{e:Aenrg}, the distributional divergence of $W$ is the $L^1(B_r)$ 
function given by
\begin{align*}
\div\, W (x) & =  
\phi'\big(\textstyle{\frac{|x|}{r}}\big)\cdot
\,\frac{x}{r\,|x|} \vartheta(x)
\Big(\frac{|\nabla u|^2}{2}x-(\nabla u\cdot x)\nabla u\Big)
 + \phi\big(\textstyle{\frac{|x|}{r}}\big)
 \Big((n-1)\vartheta(x)+(\nabla \vartheta\cdot x)\Big)\,\frac{|\nabla u(x)|^2}{2}.
\end{align*}
Being $W$ with zero trace on $\partial B_r$ we conclude that
\begin{align}\label{e:D dive}
0=\int\div\,W(x)\,\d x =&\int \phi'\Big(\textstyle{\frac{|x|}{r}}\Big)
\frac{|x|}{2\,r}\vartheta(x)\,|\nabla u(x)|^2\,\d x \notag\\
&+rE_u(r) + \frac{n-1}{2}\,D_u(r)
+\int \phi\big(\textstyle{\frac{|x|}{r}}\big)(\nabla \vartheta\cdot x)\frac{|\nabla u(x)|^2}{2}\d x.
\end{align}
Equation \eqref{e:eprime-A} follows thanks to the equalities \eqref{e:Dprime-A}, \eqref{e:D dive}, 
and the Lipschitz continuity of $\vartheta$ (cf. Lemma~\ref{l:thetalip}).

Next, we establish \eqref{e:D-A} with a similar argument. 
To this aim, consider the vector field
$V(x) := \phi\big(\frac{|x|}{r}\big)\vartheta(x)\,u(x)\,\nabla u(x)$.
Clearly, $V \in C^{\infty}(B_1\setminus B_1',\R^{n+1})$, with
\[
V(x)\cdot e_{n+1}=\phi\big(\textstyle{\frac{|x|}{t}}\big)\,\vartheta(x)\,
u(x)\,\partial_{n+1} u(x).
\]
Note that, $V\in C^{0,\sfrac12}_{\mathrm{loc}}\cap W^{1,2}_{\mathrm{loc}}(B_1^\pm\cup B_1',\R^{n+1})$
by the regularity of $u$, so that $V(x',0)\cdot e_{n+1}=0$ on $B_1'$ recalling that $u\,\partial_{n+1}u=0$ on $B_1'$. 
Thus, by taking into account that $V$ has zero trace on $\partial B_r$ and that 
$u$ minimizes \eqref{e:Aenrg}, the distributional divergence of $V$ is the 
$L^1(B_1)$ function given by
\begin{align*}
\div\, V (x) & = 
\phi'\big(\textstyle{\frac{|x|}{r}}\big)\vartheta(x)\,u(x)
\nabla u(x)\cdot\frac{x}{r\,|x|}\,+
\phi\big(\textstyle{\frac{|x|}{r}}\big)\vartheta(x)|\nabla u(x)|^2.
\end{align*}
In conclusion, \eqref{e:D-A} follows at once from the divergence theorem. 
\end{proof}

Let us now deal with the derivative of $H_u$.

\begin{proposition}\label{p:hprime}
Let $u$ be a solution to the thin obstacle problem 
\eqref{e:ob prob} in $B_1$.
There exists a non negative constant $C_{\ref{p:hprime}}$ depending on $\Lip(u)$ such that for every $x_0\in B_1'$ and for 
$\cL^1$ a.e.~$r\in(0,1-|x_0|)$,
\begin{equation}\label{e:Hprime-A}
H_u^\prime(x_0,r)=\frac nr H_u(x_0,r)
+2D_u(x_0,r)+\eps_H(x_0,r),
\end{equation} 
where $|\eps_H(x_0,r)|\leq C_{\ref{p:hprime}}H_u(x_0,r)$.
\end{proposition}
\begin{proof}
As usual we assume $x_0=\underline{0}$.
Equality \eqref{e:Hprime-A} is a consequence of \eqref{e:D-A} 
and the direct computation
\begin{align*}
H_u'(r) &= \textstyle{\frac{\d}{\d r}} \left(-r^{n} \int \phi'(|y|)\, \vartheta(r\,y)\,
\textstyle{\frac{u^2(r\,y)}{|y|}} \,\d y\right)\\
& = \textstyle{\frac{n}{r}}\,H_u(r) - \,r^{n} \int \phi'(|y|)\,\Big( 
\nabla\vartheta(r\,y){u^2(r\,y)}+
2\vartheta(r\,y) u(r\,y)\,\nabla u(r\,y)\Big)\cdot\textstyle{\frac{y}{|y|}}\,\d y\\
& \stackrel{\eqref{e:D-A}}{=} \textstyle{\frac{n}{r}}\,H_u(r) 
+ 2\, D_u(r)+\eps_H(r).
\end{align*}
where $|\eps_H(r)| \leq C_{\ref{p:hprime}}\,H_u(r)$ in view of the Lipschitz continuity 
of $\vartheta$ and \eqref{e:Acoercive}.
\end{proof}
From Proposition \ref{p:hprime} we immediately deduce
a monotonicity formula for $H_u$.
\begin{corollary}\label{c:H}
Let $u$ be a solution to the thin obstacle problem 
\eqref{e:ob prob} in $B_1$. Then, for all $x_0 \in B_1'$ and 
$0 < r_0 < r_1 < 1 - |x_0|$, we have
\begin{equation}\label{e:monotonia H}
\frac{H_u(x_0,r_1)}{r_1^{n}} = 
\frac{H_u(x_0,r_0)}{r_0^{n}}\,
\exp\left(\int_{r_0}^{r_1} \Big(2\frac{I_u(x_0,t)}{t} +\frac{\varepsilon_H(x_0,t)}{H_u(x_0,t)}\Big)\d t\right).
\end{equation}
In particular, if $A_1 \leq I(x_0,t) \leq A_2$ 
for every $t \in (r_0, r_1)$, then 
\begin{gather}
(r_0, r_1)\ni r\mapsto
e^{-C_{\ref{p:hprime}}r}\frac{H_u(x_0,r)}{r^{n+ 2A_2}}
\quad\text{is monotone decreasing},\label{e:H1}\\
(r_0, r_1)\ni r\mapsto e^{C_{\ref{p:hprime}}r}
\frac{H_u(x_0,r)}{r^{n+ 2A_1}}
\quad\text{is monotone increasing}.\label{e:H2}
\end{gather}
Moreover,  for all $0 < r < 1 - |x_0|$
\begin{align}\label{e:L2 vs H}
\frac r4\,H_u(x_0,r)\leq \int_{B_r(x_0)}|u|^2\,\d x\leq 
2\sqrt{1+L^2}\,e^{C_{\ref{p:hprime}}r}r\,H_u(x_0,r).
\end{align}
\end{corollary}
\begin{proof}
The proof of \eqref{e:monotonia H}
(and hence of \eqref{e:H1} and \eqref{e:H2})
follows immediately from the differential equation \eqref{e:Hprime-A}.

The proof of the second inequality in \eqref{e:L2 vs H} is now a direct 
consequence of \eqref{e:Acoercive} as follows
\begin{align*}
\int_{B_r(x_0)}|u|^2\,\d x
&= \sum_{k\in\N} \int_{B_{\sfrac{r}{2^k}}
\setminus B_{\sfrac{r}{2^{k+1}}}(x_0)}|u|^2\,\d x\\
&\leq \sqrt{1+L^2}\sum_{k\in\N} \frac{r}{2^k}\,H_u\big(x_0,\sfrac{r}{2^k}\big)
\leq 2\sqrt{1+L^2}\, e^{C_{\ref{p:hprime}}r}r\,H_u(x_0,r),
\end{align*}
where in the last inequality we used that
$e^{C_{\ref{p:hprime}}s}H_u(x_0,s)\leq e^{C_{\ref{p:hprime}}r}H_u(x_0,r)$ for 
$s\leq r$ by \eqref{e:H2}. The opposite inequality is elementary
in view of \eqref{e:Acoercive}. 
\end{proof}

%
%

\bibliographystyle{plain}

\end{document}